\documentclass[12pt, reqno]{amsart}
\usepackage{amssymb,amsthm,amsmath}
\usepackage[numbers,sort&compress]{natbib}
\usepackage{amssymb,amsmath}
\usepackage{amsfonts}
\usepackage{mathrsfs}
\usepackage{latexsym}
\usepackage{amssymb}
\usepackage{amsthm}
\usepackage{indentfirst}
\newtheorem{theorem}{Theorem}[section]
\newtheorem{lemma}{Lemma}[section]
\newtheorem{proposition}{Proposition}[section]
\newtheorem{definition}{Definition}[section]

\newtheorem{remark}{Remark}[section]

\begin{document}

\title[Global Weak Solutions to Non-isothermal Nematic Liquid Crystals in 2D]{Global Weak Solutions to Non-isothermal Nematic Liquid Crystals in 2D}
\author{Jinkai Li}
\address[Jinkai Li]{The Institute of Mathematical Sciences, The Chinese University of Hong Kong, Hong Kong}
\email{jklimath@gmail.com}

\author{Zhouping Xin}
\address[Zhouping Xin]{The Institute of Mathematical Sciences, The Chinese University of Hong Kong, Hong Kong}
\email{zpxin@ims.cuhk.edu.hk}



\begin{abstract}
In this paper, we prove the global existence of weak solutions to the non-isothermal nematic liquid crystal system on $\mathbb T^2$, based on a new approximate system which is different from the classical Ginzburg-Landau approximation. Local energy inequalities are employed to recover the estimates on the second order spacial derivatives of the director fields locally in time, which cannot be derived from the basic energy balance. It is shown that these weak solutions
conserve the total energy and while the kinetic and potential energies transfer to the heat energy precisely. Furthermore, it is also established that these weak solutions have at most finite many singular times at which the energy concentration occurs, and as a result, the temperature must increase suddenly at each singular time on some part of $\mathbb T^2$.
\end{abstract}

\maketitle

\section{Introduction and main results}\label{T1sec1}

The evolution of liquid crystals in $\Omega\subseteq\mathbb R^d$ is described by the following system
\begin{eqnarray}
&&u_t+(u\cdot\nabla)u+\nabla
p=\textmd{div}(S+\sigma^{nd}),\quad\textmd{div}u=0,\label{T11.1}\\
&&d_t+(u\cdot\nabla )d=\Delta d+|\nabla d|^2d,\quad|d|=1,\label{T11.2}\\
&& \theta_t+u\nabla\theta=\Delta\theta+S:\nabla u+|\Delta d+|\nabla d|^2d|^2,\label{T11.3}
\end{eqnarray}
where $u$ represents the velocity field of the flow, $d$ is the unit vector field that represents the macroscopic molecular orientation of the
liquid crystal material, $\theta$ is the absolute temperature, and $p$ denotes the pressure function. The notations $S$ and $\sigma^{\textmd{nd}}$
denote the dissipative and nondissipative part of the stress, respectively, given by
$$
S=\mu(\theta)(\nabla u+\nabla u^T),\qquad \sigma^{\textmd{nd}}=-\nabla d\odot\nabla d,
$$
where $\nabla u^T$ is the transform of $\nabla u$ and $\nabla d\odot\nabla d$ is a $d\times d$ matrix whose $(i, j)$-th entry is $\partial_id
\cdot\partial_jd$, $1\leq i, j\leq d$.

System (\ref{T11.1})--(\ref{T11.3}) becomes the non-isothermal model introduced by Feireisl-Rocca-Schimperna \cite{Feireisl1} and Feireisl-Fremond-Rocca-Schimperna \cite{Feireisl2},
where the term $|\nabla d|^2d$ is replaced by some penalty function $f(d)$ to relax the constraint $|d|=1$. The isothermal case was proposed by Lin
in \cite{Lin1} and later analyzed by Lin-Liu in \cite{Lin2}, where the penalty function $f(d)$ is considered instead of $|\nabla d|^2d$. The model
proposed in \cite{Lin1,Lin2} is a considerably simplified version of the famous Leslie-Ericksen model introduced by Ericksen \cite{Ericksen} and
Leslie \cite{Leslie} in the 1960's. Global existence of weak solutions to the system with penalty term $f(d)$ instead of $|\nabla d|^2d$ are proven
in \cite{Feireisl1,Feireisl2} and \cite{Lin2}, but the regularity and uniqueness of weak solutions is still open, at least for three dimensional case,
see Lin-Liu \cite{Lin3} for the partial regularity results. When the density of the liquid crystals is taken into consideration, the global
existence of weak solutions can be obtained in the similar way, see Jiang-Tan \cite{Jiang1} and Liu-Zhang \cite{Liu1} for the incompressible
model, see Liu-Hao \cite{Liu2} and Wang-Yu \cite{Wang} for the compressible model. In \cite{Jiang1,Liu1,Liu2,Wang}, they still consider the
penalty term $f(d)$ instead of $|\nabla d|^2d$.

Physically, $|\nabla d|^2d$ is preferred to  $f(d)$ for liquid crystal motions, and the director field $d$ must satisfy
the natural constraint $|d|=1$, see \cite{Ericksen} and \cite{Leslie}. Compared with the system with the penalty term $f(d)$, system with term $|\nabla d|^2d$ becomes more complicated technically,
due to the lack of the a priori estimates on the second order derivatives of $d$. As far as we know,  only a few papers have dealt the global existence of weak solutions to the liquid crystal system with term $|\nabla d|^2d$, see Lin-Lin-Wang \cite{Lin4}, Hong \cite{Hong1},
Hong-Xin \cite{Hong2} and two very recent papers by Wang-Wang \cite{WW} and Huang-Lin-Wang \cite{HLW}, and all of these
studies deal with the isothermal model and depend crucially on the key
idea of using the local type energy inequality, which is originally due to Struwe \cite{Struwe},
to recover the estimates on the second order derivatives of $d$ locally in time. In Lin-Lin-Wang \cite{Lin4}, global existence is proven directly to the liquid crystal
system with term $|\nabla d|^2d$, while in Hong \cite{Hong1} and Hong-Xin \cite{Hong2}, the global existence is proven by using the Ginzburg-Landau approximation.

Motivated by the works of Feireisl-Rocca-Schimperna \cite{Feireisl1} and Feireisl-Fremond-Rocca-Schimperna \cite{Feireisl2}, where the non-isothermal model with the penalty term $f(d)$ is considered, and those
of Lin-Lin-Wang \cite{Lin4}, Hong \cite{Hong1} and
Hong-Xin \cite{Hong2}, where the isothermal model with term $|\nabla d|^2d$ is considered, we consider the non-isothermal model with term $|\nabla d|^2d$
and temperature dependent viscous coefficient, which seems to be a more physical model. However, the analysis of
this model gives several difficulties and it seems hard for us to adopt the approaches of \cite{Lin4,Hong1,Hong2} directly to deal with this case. In fact, in Lin-Lin-Wang \cite{Lin4}, the Campanato space
theory is employed, which requires the H\"older continuity of the coefficients; however, in our case, due to the temperature dependence of the viscous
coefficient and the strong nonlinearity of the square terms appeared in the temperature equation, such continuity requirement is difficult to fulfill. In Hong \cite{Hong1} and Hong-Xin \cite{Hong2}, $L^\infty(H^1)$ type estimate of $(u,\nabla d)$ is used crucially in the isothermal case, which is hard to get due to
the temperature dependence of the viscous coefficient. Therefore, some different approach is needed to
deal with the non-isothermal case, and this paper is devoted to studying this issue.

We are going to prove the global existence of weak solutions to the non-isothermal liquid crystal system (\ref{T11.1})--(\ref{T11.3}) complemented with the following initial and boundary conditions
\begin{eqnarray}
&(u, d, \theta)|_{t=0}=(u_0, d_0, \theta_0), \qquad x\in\Omega=(-\pi, \pi)\times(-\pi, \pi),\label{T11.4}\\
&u, d, \theta \text{ and } p \text{ are }2\pi \text{ periodic with respect to space variables}.\label{T11.5}
\end{eqnarray}

To establish the global existence of weak solutions, we introduce the following elaborate approximation system which will be used throughout this paper
\begin{eqnarray}
&u_t+(u\cdot\nabla)u+\nabla p=\text{div}(S_N-\nabla d\odot\nabla d),\qquad\text{div}u=0,\label{T11.6}\\
&d_t+(u\cdot\nabla)d=\Delta d+\chi_M(|\nabla d|^2)d,\label{T11.7}\\
&\theta_t+u\cdot\nabla\theta=\Delta\theta+S_N:\nabla u+|\Delta d+\chi_M(|\nabla d|^2)d|^2\label{T11.8},
\end{eqnarray}
where
\begin{equation}\label{T11.9}
S_N=\mu(\theta)(\nabla u+\nabla u^T)+\frac{1}{N}|\nabla u|^{\frac{2}{9}}\nabla u
\end{equation}
and $\chi_M$ is a cut off function (see Section \ref{T1sec3} for the exact definition of $\chi_M$).
We note that, different from the Ginzburg-Landau approximate system, this
system is not a singular type approximation of the original system, and thus one can derive
its estimates more easily.
Moreover, the uniform $L^2(H^2)$ estimates on $d$ of this system can be derived
by the local type energy inequality without any use of the $L^\infty(H^3)$ type
energy on $d$, while for the Ginzburg-Landau system, one need the $L^\infty(H^3)$ type
energy to derive the uniform $L^2(H^2)$ estimates, see Hong \cite{Hong1} and Hong-Xin \cite{Hong2} for the details. We also note that $L^2(H^2)$ type estimates on $d$
is enough for us to take the limit of the
approximate system to obatin the weak solutions to the original system.

To state the main result of this paper, we introduce some notations. For $T>0$, set $Q_T=\Omega\times(0, T)$. For $1\leq p\leq \infty$ and $k\in\mathbb N$, $L^p(\Omega)$ is the Lebesgue space, $L^2_\sigma(\Omega)$ is the space consisting of all divergence free functions in $L^2(\Omega)$, and $W^{k,p}(\Omega)$
is the standard Sobolev space. When $p=2$, $H^k(\Omega)$=$W^{k,2}(\Omega)$. We use $W^{k,p}_{\text{per}}(\Omega)$
or $H^k_{\text{per}}(\Omega)$ to denote the spaces consisting of all $2\pi$ periodic functions in $W^{k,p}(\Omega)$ or $H^1(\Omega)$, respectively.
More precisely,
$$
W^{k,p}_{\text{per}}(\Omega)=\left\{u\left|\tilde u\in W^{k,p}_{\text{loc}}(\mathbb R^N), \text{ where }\tilde u\text{ is the } 2\pi \text{ periodic extension of } u\right.\right\}.
$$
The dual space of $W^{1,p}_{\text{per}}(\Omega)$ is denoted by $W_{\text{per}}^{-1,p'}(\Omega)$ for $1\leq p<\infty$, where $\tfrac{1}{p}+\tfrac{1}{p'}=1$.
The same symbol will be used  for function $u$ and its $2\pi$ periodic extension.

The definition of weak solutions to the system (\ref{T11.1})--(\ref{T11.5}) is given by the following definition.

\begin{definition}\label{T1def1.1}
Let $T$ be a finite positive number. The triple $(u, d, \theta)$ is called a weak solution to the system (\ref{T11.1})--(\ref{T11.3}) with (\ref{T11.4}--(\ref{T11.5}) in $Q_T=\Omega\times(0, T)$, if all the following three conditions hold true

(i) $u, d$ and $\theta$ are $2\pi$ periodic, and
\begin{eqnarray*}
&&u\in L^\infty(0, T; L^2(\Omega))\cap L^2(0, T; H^1(\Omega)),\quad \textmd{div} u =0,\\
&&d\in L^\infty(0, T; H^1(\Omega; S^2)),\quad\Delta d+|\nabla d|^2d\in L^2(Q_T),\\
&&\theta\in L^\infty(0, T; L^1(\Omega))\cap L^q(0, T; W^{1, q}(\Omega)), \quad q\in (1, \tfrac{4}{3}),
\end{eqnarray*}

(ii) equation (\ref{T11.2}) is satisfied a.e. in $Q_T$, and $d$ satisfies the initial condition (\ref{T11.4}),

(iii) the following two integral identities hold true
\begin{eqnarray*}
&&\int_0^T\int_\Omega[(S+\sigma^{nd}-u\otimes u):\nabla\varphi-u\varphi_t]dxdt=\int_\Omega u_0\varphi(x,0)dx
\end{eqnarray*}
for any $2\pi$ periodic functions $\varphi\in C^\infty(\overline\Omega\times[0, T))$ with $\text{div}\varphi=0$, and
\begin{align*}
&\int_0^T\int_\Omega\nabla\theta\cdot\nabla\phi dxdt-\int_\Omega\theta_0\phi(x,0)dx\\
=&\int_0^T\int_\Omega\left[\theta\phi_t+\left(S:\nabla u+|\Delta d+|\nabla d|^2d|^2-u\cdot\nabla\theta\right)\phi\right]dxdt
\end{align*}
for any $2\pi$ periodic functions $\phi\in C^\infty(\overline\Omega\times[0, T))$.
\end{definition}

Though out this paper, it is always assumed that the coefficient $\mu$ satisfies
\begin{equation}\label{T11.10}
\mu(\theta) \text{ is continuous in }\theta,\quad\underline\mu\leq\mu(\theta)\leq\overline\mu
\end{equation}
for two positive constants $\underline\mu$ and $\overline\mu$.

The main result of the present paper is the following theorem.

\begin{theorem}\label{T1them1.2}
Assume that (\ref{T11.10}) holds true, and that the $2\pi$ periodic functions $u_0, d_0, \theta_0$ satisfy
$$
u_0\in L^2_\sigma(\Omega),\quad d_0\in H^1(\Omega;S^2),\quad \theta_0\in L^1(\Omega),\quad\inf_{x\in\Omega}\theta_0(x)\geq{\underline{\theta}}_0
$$
for some positive constant $\underline{\theta}_0$.

Then there is a weak solution to the system (\ref{T11.1})--(\ref{T11.5}) in $Q_T$ for any finite $T$, such that
$$
\int_\Omega\left(\frac{|u(t)|^2}{2}+\frac{|\nabla d(t)|^2}{2}+\theta(t)\right)dx=\int_\Omega\left(\frac{|u_0|^2}{2}+\frac{|\nabla d_0|^2}{2})+\theta_0\right)dx
$$
for any $t\in[0, T]$, and the following entropy inequality holds true
$$
\partial_t\theta^\alpha+\textmd{div}(u\theta^\alpha)\geq\Delta\theta^\alpha+\alpha\theta^{\alpha-1}(S:\nabla u+|\Delta d+|\nabla d|^2d|^2)+\alpha(1-\alpha)\theta^{\alpha-2}|\nabla\theta|^2
$$
in $\mathcal D'(Q_T)$ for any $\alpha\in(0, 1)$.

Moreover, there are at most finite many "singular" times $0<T_1<\cdots<T_N<\infty$ characterized by
$$
\sup_{(x, t)\in\overline{Q_{T_i}}}\int_{B_r(x)}(|u(t)|^2+|\nabla d(t)|^2)dy\geq\varepsilon_0^2,\quad 1\leq i\leq N,\quad\forall r>0,
$$
where $\varepsilon_0$ is a positive constant depending only on $\int_\Omega(|u_0|^2+|\nabla d_0|^2)dx$,
such that
\begin{align*}
\int_\Omega(|u(T_i)|^2+|\nabla d(T_i)|^2)dx=&2\int_{T_i}^t\int_\Omega(S:\nabla u+|\Delta d+|\nabla d|^2d|^2)dxds\\
&+\int_\Omega(|u(t)|^2+|\nabla d(t)|^2)dx
\end{align*}
for all $t\in[T_i, T_{i+1}), 0\leq i\leq N$ (here we have set $T_0=0$ and $T_{N+1}=\infty$), and
\begin{eqnarray*}
&&d\in L^2(T_i, T; H^2(\Omega))\cap C^{\beta,\beta/2}(\overline\Omega\times(T_i, T_{i+1}),\quad p\in L^2(\Omega\times(T_i, T))
\end{eqnarray*}
for any $T\in[T_i, T_{i+1}), 0\leq i\leq N$ and some $\beta\in(0, 1)$.
\end{theorem}

\begin{remark}
(i) Theorem \ref{T1them1.2} generalizes the results of Lin-Lin-Wang \cite{Lin4}, Hong \cite{Hong1} and Hong-Xin \cite{Hong2} to the non-isothermal model with
temperature dependent viscous coefficient.

(ii) The Cauchy problem and the initial-boundary value problem (imposing slip boundary condition on the velocity) to the system (\ref{T11.1})--(\ref{T11.3}) can be dealt with similarly.
\end{remark}

\begin{remark}
Theorem \ref{T1them1.2} implies that the total energy is conserved and the energy concentration happens at each singular time. Due to this energy
conservation and possible concentration, the kinetic energy $\int_\Omega|u|^2dx$ and the potential energy $\int_\Omega|\nabla d|^2dx$
transfer to nothing but the heat energy, and the temperature increases suddenly at each singular time. It is noted that, for the isothermal model,
the summation of the kinetic energy and the potential energy also decreases, however,
for this case, it's unclear what kind of energy they transfer to.
\end{remark}

To prove Theorem \ref{T1them1.2}, we consider the corresponding initial and periodic
boundary value problems of the approximate system (\ref{T11.6})--(\ref{T11.8}), expecting that the solutions of this system converge
to the ones of the original system. The main issue is to derive the $L^2(H^2)$ estimates (uniform in $M$ and $N$) on the director fields. Due to the nonlinearity
term $\chi_M(|\nabla d|^2)d$ appeared in (\ref{T11.7}), it seems difficult to obtain desired
estimates on the second derivatives of $d$ directly from the basic energy estimates. Hence, it is hard to show that
the approximate solutions converge to a weak solution of the original system. To overcome this difficulty, similar to Lin-Lin-Wang
\cite{Lin4}, Hong \cite{Hong1} and Hong-Xin \cite{Hong2}, we make use of the local energy inequality and recover the estimates on the second order derivatives of $d$ locally in time. More precisely, it will be shown that the $L^2(Q_T)$
norm of $\Delta d$ is bounded by the initial energy for
some small time $T$, where $T$ depends only on the size of the initial basic energy and the $L^2$ integral continuity of $(u_0, \nabla d_0)$. As a
result, the approximate solutions exist in a common time interval in which they have uniform estimates on the second order derivatives of the director
fields. We remark that the initial data $\theta_0$ has no influence on such $T$. Once we recover the $L^2$ bounds of $\Delta d$ for the approximate
solutions, we can take the limit to obtain the weak solution locally in time. Due to this local existence result, we can extend such weak solution to
the first singular time, which is characterized by the appearance of energy concentration. In other words, one can show that the energy functional
$$
\tilde{\mathcal E_1}(t)=\int_\Omega(|u(t)|^2+|\nabla d(t)|^2)dx
$$
losses its energy of order at least $\varepsilon_0^2$ at the singular time, for some positive constant $\varepsilon_0$ depending only on the
initial data $u_0$ and $d_0$. Due to this and the fact that the total energy is conserved, there are at most finite many such singular
times, and thus one obtains the global weak solutions.

In proving the local existence of weak solutions, we will employ four level approximations: the  first one is a standard Faedo-Galerkin approximation of
the system (\ref{T11.6})--(\ref{T11.8}) and obtain the approximate solutions
$(u_{n,M,N}, d_{n, M, N}, \theta_{n, M, N})$ on any finite time interval $[0,T]$; next, we take the limit as $n$ goes to infinity to get a
sequence of approximate solutions $\{(u_{M,N}, d_{M,N},\theta_{M,N})\}$ for any finite time $T$, which possesses uniform estimates on
$\Delta d_{M,N}$ in a common time interval $[0, T_1]$; then, by takeing the limit $M\rightarrow\infty$, we obtain a sequence of approximate
weak solutions $\{(u_N, d_N, \theta_N)\}$ which has uniform estimates on $\Delta d_N$ on a common time interval $[0, T_1]$; and finally, we
study the limit $N\rightarrow\infty$ to get a local weak solution to the original system.

It should be remarked that in this approach,
though the approximate stress $S_N$ plays no role in the Faedo-Galerkin scheme and in the steps $n\rightarrow\infty$
and $N\rightarrow\infty$, it indeed plays an important role in the step $M\rightarrow\infty$, where the ingredient
$\frac{1}{N}|\nabla u|^{\frac{2}{9}}\nabla u$ in $S_N$
provides us the higher integrability of $u_{M,N}$ than $L^4(Q_{T_1})$, which is used to prove the $C^{\beta,\beta/2}$ estimates on
$\nabla d_{M,N}$ (see Lemma \ref{T1lem4.3} for the details.)

The rest of this paper is arranged as follows: in Section \ref{T1sec2}, some preliminaries, including some inequalities and strong
convergence lemmas, are stated, which will be used though out this paper; in Section \ref{T1sec3}, we carry out the Faedo-Galerkin approximation and obtain the
approximate solution for fixed $n, M$ and $N$ for all finite time $T$; in Section \ref{T1sec4}, we study the limit $n\rightarrow\infty$ for fixed
$M$ and $N$ to establish the approximate solution for all finite time; Section \ref{T1sec5} is employed to study the limit $M\rightarrow\infty$ for
fixed $N$ to get the local weak solution for approximate system, where the local energy inequalities play a key role; in Section \ref{T1sec6}, we
establish the local weak solution to the system (\ref{T11.1})--(\ref{T11.5}), while the global weak solution is finally obtained in the last
section, Section \ref{T1sec7}, via extending the local weak solution to be a global one.

\section{Preliminaries}\label{T1sec2}

In this section, we give some lemmas, including some inequalities and some compactness and strong convergence lemmas, which will be used in the rest of this paper.

The first lemma is a version of the Ladyzhenskaya inequality.

\begin{lemma}\label{T1lem2.1}
(See Strune \cite{Struwe} Lemma 3.1) There exist $C_0>0$ and $R_0>0$ depending only on $\Omega\subseteq\mathbb R^2$ such that for any $T>0$, any $0<R<R_0$, and any $u\in L^\infty(0, T; L^2(\Omega))\cap L^2(0, T; H^1(\Omega)),$ it holds that
$$
\int_{Q_T}|u|^4dxdt\leq C_0\left(\sup_{(x, t)\in\overline{Q_T}}\int_{\Omega\cap B_R(x)}|u(t)|^2dy\right)\int_{Q_T}\left(|\nabla u|^2+\frac{|u|^2}{R^2}\right)dxdt.
$$
\end{lemma}

The next two lemmas concern compactness.

\begin{lemma}\label{T1lem2.2}
(See Simon \cite{Simon} Theorem 3) Assume that $X$ and $B$ are two Banach spaces with $X\hookrightarrow\hookrightarrow B$. Let $F$ be a bounded subset of $L^p(0, T; X)$ where $1\leq p\leq \infty$, and assume that
$$
\|\tau_h f-f\|_{L^p(0, T-h; B)}\rightarrow0\quad\text{as }h\rightarrow0,\text{ uniformly for }f\in F,
$$
where $\tau_hf(t)=f(t+h)$.

Then $F$ is relatively compact in $L^p(0, T; B)$ (and in $C([0, T]; B)$ if $p=\infty$).
\end{lemma}

\begin{lemma}\label{T1lem2.3}
(See Simon \cite{Simon} Corollary 4) Assume that $X, B$ and $Y$ are three Banach spaces, with $X\hookrightarrow\hookrightarrow B\hookrightarrow Y.$ Then the following hold true

(i) If $F$ is a bounded subset of $L^p(0, T; X)$, $1\leq p<\infty$, and
$\frac{\partial F}{\partial t}=\left\{\frac{\partial f}{\partial t}\Big|f\in F\right\}$ is bounded in $L^1(0, T; Y)$. Then $F$ is relatively compact in $L^p(0, T; B)$;

(ii) If $F$ is bounded in $L^\infty(0, T; X)$ and $\frac{\partial F}{\partial t}$ is bounded in $L^r(0, T; Y)$ where $r>1$. Then $F$ is relatively compact in $C([0, T]; B)$.
\end{lemma}

The next one is a version of the well known Korn's inequality.

\begin{lemma}\label{T1lem2.4}
(See Wang \cite{WangL} Theorem 1) Let $\Omega\subseteq\mathbb R^n$ ($n=2, 3$) be a smooth bounded domain or polygon, then it holds that
$$
\|u\|_{W^{1, r}(\Omega)}^r\leq C(\Omega)\int_\Omega(|\nabla u+\nabla u^T|^2+|u|^r)dx
$$
for any $u\in W^{1, r}(\Omega)$ with $1<r<\infty$.
\end{lemma}

The following lemma deals with strong convergence in suitable spaces:

\begin{lemma}\label{T1lem2.7}
Let $2\leq p<\infty, 1<q<\infty$ and assume functions $u_n$ and $v_n$ satisfy
\begin{eqnarray*}
&&u_n\rightarrow u\mbox{ weakly in }L^p(\Omega),\quad v_n\rightarrow v\mbox{ weakly in }L^q(\Omega),\\
&& |u_n|^{p-2}u_n\rightarrow f\mbox{ weakly in }L^{\frac{p}{p-1}}(\Omega), \\
&&\int_\Omega(|u_n|^p+|v_n|^q)dx\rightarrow\int_\Omega(f\cdot u+|v|^q)dx.
\end{eqnarray*}
Then
$$
u_n\rightarrow u\mbox{ strongly in }L^p(\Omega),\quad v_n\rightarrow v\mbox{ strongly in }L^q(\Omega).
$$
\end{lemma}

\begin{proof}
Making use of an inequality of the form (see e.g. Damascelli \cite{Damascelli} Lemma 2.1)
\begin{equation}
(|\xi|^{p-2}\xi-|\eta|^{p-2}\eta)\cdot(\xi-\eta)\geq C|\xi-\eta|^p,\quad\forall\xi,\eta\in\mathbb R^N,\label{A.111}
\end{equation}
one can deduce easily from the assumption that
\begin{equation*}
\int_\Omega f\cdot udx\leq\liminf_{n\rightarrow\infty}\int_\Omega|u_n|^pdx.
\end{equation*}
This, together with the weak lower semi-continuity of norms, implies
\begin{align*}
\int_\Omega(f\cdot u+|v|^q)dx\leq&\liminf_{n\rightarrow\infty}\int_\Omega|u_n|^pdx+\liminf_{n\rightarrow\infty}\int_\Omega|v_n|^qdx\\
\leq&\limsup_{n\rightarrow\infty}\int_\Omega|u_n|^pdx+\liminf_{n\rightarrow\infty}\int_\Omega|v_n|^qdx\\
\leq&\limsup_{n\rightarrow\infty}\int_\Omega(|u_n|^p+|v_n|^q)dx=\int_\Omega(f\cdot u+|v|^q)dx,
\end{align*}
and similarly
\begin{align*}
\int_\Omega(f\cdot u+|v|^q)dx\leq&\liminf_{n\rightarrow\infty}\int_\Omega|u_n|^pdx+\liminf_{n\rightarrow\infty}\int_\Omega|v_n|^qdx\\
\leq&\liminf_{n\rightarrow\infty}\int_\Omega|u_n|^pdx+\limsup_{n\rightarrow\infty}\int_\Omega|v_n|^qdx\\
\leq&\limsup_{n\rightarrow\infty}\int_\Omega(|u_n|^p+|v_n|^q)dx=\int_\Omega(f\cdot u+|v|^q)dx.
\end{align*}
These yield
\begin{eqnarray*}
&&\lim_{n\rightarrow\infty}\int_\Omega|u_n|^pdx=\liminf_{n\rightarrow\infty}\int_\Omega|u_n|^pdx=\limsup_{n\rightarrow\infty}\int_\Omega|u_n|^pdx,\\
&&\lim_{n\rightarrow\infty}\int_\Omega|v_n|^qdx=\liminf_{n\rightarrow\infty}\int_\Omega|v_n|^qdx=\limsup_{n\rightarrow\infty}\int_\Omega|v_n|^qdx,
\end{eqnarray*}
and therefore,
\begin{equation*}
\lim_{n\rightarrow\infty}\int_\Omega|u_n|^{p}u_ndx\geq\int_\Omega f\cdot udx,\quad\lim_{n\rightarrow\infty}\int_\Omega|v_n|^{q}dx\geq\int_\Omega |v|^qdx.
\end{equation*}
This, combined with the assumption
$$
\lim_{n\rightarrow\infty}\int_\Omega(|u_n|^p+|v_n|^q)dx=\int_\Omega(f\cdot u+|v|^q)dx,
$$
implies
\begin{equation*}
\lim_{n\rightarrow\infty}\int_\Omega|u_n|^{p}u_ndx=\int_\Omega f\cdot udx,\qquad\lim_{n\rightarrow\infty}\int_\Omega|v_n|^{q}dx=\int_\Omega |v|^qdx.
\end{equation*}
Due to this, one can use inequality (\ref{A.111}) again to obtain the strong convergence of $u_n$ to $u$ easily, while the strong convergence of $v_n$ to
$v$ follows from the uniform convexity of the space $L^q(\Omega)$ (see e.g. Adams-Fournier \cite{Adams} Theorem 2.39.) This completes the proof.
\end{proof}

We also need the following elementary lemma.

\begin{lemma}\label{T1lem2.8}
(See Section 4 of Schoen and Uhlenbeck \cite{Schoen}) Let $M^2$ be a compact surface with possibly empty $C^1$ boundary. Let $N$ be a compact manifold without boundary. Then $C^\infty(M, N)$ is dense in $H^1(M, M)$.
\end{lemma}

\section{Faedo-Galerkin scheme}
\label{T1sec3}

In this section, we carry out  the Faedo-Galerkin approximation to the system (\ref{T11.6})--(\ref{T11.8}).
For $M>0$, we introduce a cut off function $\chi_M$ given by
$$
\chi_M(s)=M\chi(M^{-1}s),\qquad s\geq0,
$$
where
$$
\chi(s)=\left\{\begin{array}{lr}s,&0\leq s\leq 1,\\
1,&s>1.\end{array}\right.
$$
For $N>0$, set
$$
S_N=\mu(\theta)(\nabla u+\nabla u^T)+\frac{1}{N}|\nabla u|^{\frac{2}{9}}\nabla u.
$$
For $n\in\mathbb N$, set $X_n=\text{span}\{e^{ikx}\}_{k=-n}^{n}$
and
$$
X_{n,\text{div}}=\{u\in X_n|\text{div}u=0\}.
$$

For given $M$, $N$ and $n$, we look for a solution $(u, d, \theta)$ with $u\in C([0, T_n]; X_{n,\textmd{div}})$ to the following system
\begin{eqnarray}
&&\frac{d}{dt}\int_\Omega u\cdot wdx+\int_\Omega (S_N-\nabla d\odot\nabla d):\nabla w dx\nonumber\\
&&~~~~~~~~~+\int_\Omega(u\cdot\nabla)u\cdot wdx=0,\qquad\forall w\in X_{n, \textmd{div}},\label{T12.6}\\
&&d_t+(u\cdot \nabla)d=\Delta d+\chi_M(|\nabla d|^2)d,\label{T12.7}\\
&&\theta_t+u\cdot\nabla\theta=\Delta\theta+S_N:\nabla u+|\Delta d+\chi_M(|\nabla d|^2)d|^2,\label{T12.8}
\end{eqnarray}
supplemented with the initial and boundary conditions (\ref{T11.4})--(\ref{T11.5}).

Suppose that the functions $u_0, d_0$ and $\theta_0$ are $2\pi$ periodic and satisfy
\begin{equation}\label{T12.12}
u_0\in X_{n, \textmd{div}}, \quad d_0\in C^\infty(\bar\Omega), \quad |d_0|\leq 1, \quad \theta_0\in C^\infty(\bar\Omega), \quad\inf_{x\in\bar\Omega}\geq\underline\theta_0
\end{equation}
for some positive constant $\underline\theta_0$.

By the standard parabolic theory, for given $u\in C([0, T]; X_{n,\textmd{div}})$, one can solve (\ref{T12.7}) and (\ref{T12.8}) uniquely with initial and
boundary conditions (\ref{T11.4})--(\ref{T11.5})
under assumption (\ref{T12.12}). Besides, it follows from the maximal principle for parabolic equations that
\begin{equation}\label{T12.14}
\inf_{(x,t)\in Q_T}\theta\geq\underline\theta_0,\quad\text{and}\quad\sup_{(x,t)\in Q_T}|d|\leq1.
\end{equation}
Thus, the functions $\theta$ and $d$ appeared in (\ref{T12.6}) can be uniquely expressed through (\ref{T12.7})--(\ref{T12.8}) with (\ref{T11.4})--(\ref{T11.5}) for given $u$. Accordingly, problem (\ref{T12.6})--(\ref{T12.8}) with (\ref{T11.4})--(\ref{T11.5}) can be uniquely solved via the
standard fixed point theory, at least in a (possibly) short time interval $(0, T_n)$. Taking $w=u(t)$ in (\ref{T12.6}), multiplying (\ref{T12.7})
by $-\Delta d$ and integrating over $\Omega$, summing the resulting equations up, noticing that
$$
\textmd{div}(\nabla d\odot\nabla d)=\nabla\left(\tfrac{|\nabla d|^2}{2}\right)+\Delta d\cdot\nabla d,
$$
and using (\ref{T12.14}), we obtain the inequality
\begin{align}
&\frac{d}{dt}\int_\Omega(|u(t)|^2+|\nabla d(t)|^2)dx+\int_\Omega(2S_N:\nabla u+|\Delta d|^2)dx\nonumber\\
\leq&\int_\Omega|\chi_M(|\nabla d|^2)|^2dx\leq 4\pi M^2.\label{T12.15}
\end{align}
This implies immediately that the existence time $T_n$ can be taken as any finite time $T$. Hence, for any given $M, N$ and $n$, we have
established the existence of solution $(u, d, \theta)$ to the approximate system (\ref{T12.6})--(\ref{T12.8}) with (\ref{T11.4})--(\ref{T11.5}).

The pressure $p$ is determined through
$$
p=\Delta^{-1}\textmd{div}\textmd{div}(S_N-\nabla d\odot\nabla d-u\otimes u),
$$
more preciously, $p$ is defined as the unique solution to the problem
\begin{equation}\label{T12.16}
\left\{\begin{array}{l}
\Delta p=\textmd{div}\textmd{div}(S_N-\nabla d\odot\nabla d-u\otimes u)\quad \mbox{in }\Omega,\\
p \mbox{ is }2\pi \mbox{ periodic }, \quad \int_\Omega pdx=0.
\end{array}\right.
\end{equation}
With $p$ thus defined, we can rewrite (\ref{T12.6}) as
\begin{equation}\label{T12.17}
\frac{d}{dt}\int_\Omega u\cdot wdx+\int_\Omega [(S_N-\nabla d\odot\nabla d-u\otimes u):\nabla w-p\textmd{div}w]dx=0,\quad\forall w\in X_n.
\end{equation}
In fact, for any $w\in X_n$, decomposed as $w=w_1+\nabla q$ with $w_1\in X_{n, \textmd{div}}$ and $q\in X_n$,
it follows from (\ref{T12.6}) and (\ref{T12.16}) that
\begin{equation*}
\frac{d}{dt}\int_\Omega u\cdot w_1dx+\int_\Omega [(S_N-\nabla d\odot\nabla d-u\otimes u):\nabla w_1dx=0,
\end{equation*}
and
\begin{align*}
-\int_\Omega p\textmd{div}wdx=&-\int_\Omega p\Delta qdx=-\int_\Omega\Delta pqdx\\
=&-\int_\Omega\textmd{div}\textmd{div}(S_N-\nabla d\odot\nabla d-u\otimes u)qdx\\
=&\int_\Omega\textmd{div}(S_N-\nabla d\odot\nabla d-u\otimes u)\cdot\nabla qdx.
\end{align*}
Summing up the above two identities and noticing that $\int_\Omega u\cdot w_1dx=\int_\Omega u\cdot wdx$, we obtain (\ref{T12.17}).

Collecting all the above statements, we have proven the following

\begin{proposition}\label{T1prop2.1}
Under the assumption (\ref{T12.12}), for any given $M>0$, $N>0$ and $n\in\mathbb N$, there is a unique solution $(u, d, \theta, p)$ to the system (\ref{T12.6})--(\ref{T12.8}) with (\ref{T11.4})--(\ref{T11.5}), such that (\ref{T12.14})--(\ref{T12.17}) hold true.
\end{proposition}

\section{The limit $n\rightarrow\infty$}\label{T1sec4}

In the previous section, we have proven that, for any given $M>0, N>0$ and $n\in\mathbb N$, the problem (\ref{T12.6})--(\ref{T12.8})
with (\ref{T11.4})--(\ref{T11.5}) has a unique solution $(u_{M,N,n}, d_{M,N,n}, \theta_{M,N,n}, p_{M,N,n})$.
In this section, we study the limit $n\rightarrow\infty$ for fixed $M$ and $N$ to prove the solvability of the system (\ref{T11.6})--(\ref{T11.8})
with the initial and boundary conditions (\ref{T11.4})--(\ref{T11.5}). More precisely, we will prove the following result.

\begin{proposition}\label{T1prop3.1}
Assume that the $2\pi$ periodic functions $u_0, d_0$ and $\theta_0$ satisfy
$$
u_0\in L^2_\sigma(\Omega),\quad d_0\in H^1(\Omega),\quad\sup_{x\in\Omega}|d_0|\leq 1, \quad\theta_0\in L^1(\Omega),\quad\inf_{x\in\Omega}\theta_0\geq\underline\theta_0
$$
for some positive constant $\underline\theta_0$.

Then for any $M>0$ and $N>0$, the system (\ref{T11.6})--(\ref{T11.8}) with (\ref{T11.4})--(\ref{T11.5})
has a weak solution $(u, d, \theta, p)$ in $Q_T$ for any $T>0$, such that
\begin{eqnarray}
&\inf_{(x, t)\in Q_T}\theta\geq\underline\theta_0, \qquad\sup_{(x,t)\in Q_T}|d|\leq 1,\nonumber\\
&\int_\Omega\theta(t)dx\leq Q_{M,N}(t),\qquad \int_0^t\int_\Omega|\nabla\theta|^qdxd\tau\leq C(q, T)[Q_{M,N}(t)]^q,\label{T13.0-6}
\end{eqnarray}
for $t\in[0, T]$ and $q\in(1, \tfrac{4}{3})$, where
$$
Q_{M,N}(t)=\int_\Omega\theta_0dx+\int_0^t\int_\Omega(S_N:\nabla u+|\Delta d+\chi_M(|\nabla d|^2)d|^2)dxd\tau.
$$
\end{proposition}

\begin{proof}
Choose $2\pi$ periodic funtions $\{(u_{0,n}, d_{0,n}, \theta_{0,n})\}\subseteq C^\infty(\bar\Omega)$, such that
$$
u_{0,n}\in X_{n,\textmd{div}}, \quad\sup_{x\in\Omega}|d_{0,n}|\leq1, \quad\inf_{x\in\Omega}\theta_{0,n}\geq\underline\theta_0,
$$
and
$$
u_{0,n}\rightarrow u_0\mbox{ in }L^2(\Omega),\quad d_{0,n}\rightarrow d_0\mbox{ in }H^1(\Omega),\quad \theta_{0,n}\rightarrow\theta_0\mbox{ in }L^1(\Omega).
$$

By Proposition \ref{T1prop2.1}, for any $n\in\mathbb N$, there is a unique solution $(u_n, d_n, \theta_n, p_n)$ to the system (\ref{T12.6})--(\ref{T12.8}) with the periodic boundary condition (\ref{T11.5}) and initial data $(u_{0, n}, d_{0,n}, \theta_{0,n})$, such that
$$
\inf_{(x, t)\in Q_T}\theta_n\geq\underline\theta_0,\quad\sup_{(x, t)\in Q_T}|d_n|\leq 1
$$
and
\begin{equation}
\sup_{0\leq t\leq  T}\int_\Omega(|u_n|^2+|\nabla d_n|^2)dx+\int_0^T\int_\Omega(|\nabla u_n|^{\frac{20}{9}}+|\nabla^2d_n|^2)dxdt\leq C(M,N,T).\label{T13.1}
\end{equation}
It follows from the Gagliado-Nirenberg inequality and (\ref{T13.1}) that
\begin{equation}
\int_0^T\int_\Omega(|u_n|^{\frac{40}{9}}+|\nabla d_n|^4)dxdt\leq C(M, N, T). \label{T13.3}
\end{equation}
These two inequalities, together with equation (\ref{T12.7}), lead to
\begin{align}
\|\partial_td_n\|_{L^2(Q_T)}\leq& C(\|u_n\|_{L^4(Q_T)}\|\nabla d_n\|_{L^4(Q_T)}+\|\Delta d_n+\chi_M(|\nabla d_n|^2)d_n\|_{L^2(Q_T)})\nonumber\\
\leq &C(M,N,T).\label{T13.3-1}
\end{align}
Applying elliptic estimates to (\ref{T12.16}), it follows from (\ref{T13.1}) and (\ref{T13.3}) that
\begin{align}
&\int_0^T\int_\Omega|p_n|^{\frac{20}{11}}dxdt\leq C\int_0^T\int_\Omega\left(|S_N|^{\frac{20}{11}}+|u_n|^{\frac{40}{11}}+|\nabla d_n|^{\frac{40}{11}}\right)dxdt\nonumber\\
\leq &C\int_0^T\int_\Omega\left(|u_n|^{\frac{40}{11}}+|\nabla d_n|^{\frac{40}{11}}+N^{-\frac{20}{11}}|\nabla u_n|^{\frac{20}{9}}+|\nabla u_n|^{\frac{20}{11}}\right)dxdt\nonumber\\
\leq&C\int_0^T\int_\Omega\left(1+|u_n|^{\frac{40}{9}}+|\nabla d_n|^4+|\nabla u_n|^{\frac{20}{9}}\right)dxdt\leq C(M, N, T).\label{T13.4}
\end{align}

By the aid of (\ref{T13.1}) and (\ref{T13.3}), it follows from identity (\ref{T12.6}) that
\begin{align*}
&\int_\Omega(u_n(x,t+h)-u_n(x,t))\cdot wdx\\
=&\int_t^{t+h}\int_\Omega(\nabla d_n\odot\nabla d_n+u_n\otimes u_n-S_N):\nabla wdxd\tau\\
\leq&C\int_t^{t+h}\int_\Omega(|\nabla d_n|^2+|u_n|^2+|\nabla u_n|^{\frac{11}{9}}+|\nabla u_n|)|\nabla w|dxd\tau\\
\leq&Ch^{\frac{1}{2}}\|\nabla w\|_{L^2(\Omega)}\left(\int_t^{t+h}\int_\Omega(|\nabla d_n|^4+|u_n|^4+|\nabla u_n|^2)dxd\tau\right)^{\frac{1}{2}}\\
&+Ch^{\frac{9}{20}}\|\nabla w\|_{L^{\frac{20}{9}}(\Omega)}\left(\int_t^{t+h}\int_\Omega|\nabla u_n|^{\frac{20}{9}}dxd\tau\right)^{\frac{9}{20}}\\
\leq&C(M,N,T)\left(\|\nabla w\|_{L^2(\Omega)}h^{\frac{1}{2}}+\|\nabla w\|_{L^{\frac{20}{9}}(\Omega)}h^{\frac{9}{20}}\right)\\
\leq&C(M,N,T)\|\nabla w\|_{L^{\frac{20}{9}}(\Omega)}h^{\frac{9}{20}}
\end{align*}
for any $w\in X_{n,\text{div}}$ and $t\in[0, T-h]$, from which, by taking $w=u_n(x,t+h)-u_n(x,t)$, one arrives
\begin{equation}
\|\tau_hu_n-u_n\|_{L^2(0,T-h; L^2(\Omega))}^2\leq C(M,N,T)h^{\frac{9}{20}}\|\nabla u_n\|_{L^{\frac{20}{9}}(Q_T)}\leq C(M,N,T)h^{\frac{9}{20}},\label{T13.4-1}
\end{equation}
where $\tau_hu_n(t)=u_n(t+h)$.

Integrating (\ref{T12.8}) over $\Omega$ yields
\begin{equation}\label{T13.4-2}
\int_\Omega\theta_n(t)dx=Q_{M, N, n}(t),
\end{equation}
where
\begin{align*}
Q_{M,N,n}(t)=&\int_\Omega\theta_{0,n}dx+\int_0^t\int_\Omega\left(\frac{\mu(\theta_n)}{2}|\nabla u_n+\nabla u_n^T|^2\right.\\
&\left.+\frac{1}{N}|\nabla u_n|^{\frac{20}{9}}+|\Delta d_n+\chi_M(|\nabla d_n|^2)d_n|^2\right)dxd\tau.
\end{align*}
Multiplying (\ref{T12.8}) by $H'(\theta)$ leads to
\begin{equation*}
\partial_tH(\theta)+u\nabla H(\theta)=\Delta H(\theta)+H'(\theta)(S_N:\nabla u+|\Delta d+\chi_M(|\nabla d|^2)d|^2)-H''(\theta)|\nabla\theta|^2,
\end{equation*}
from which, by taking $H(\theta)=\theta^\alpha$ with $\alpha\in(0, 1)$ and integrating over $\Omega$, one arrives
\begin{align*}
&\int_0^t\int_\Omega\left|\nabla\theta_n^{\frac{\alpha}{2}}\right|^2dxd\tau\leq C(\alpha)\int_\Omega\theta_n^\alpha(t)dx\\
\leq& C(\alpha)\left(\int_\Omega\theta_n(t)dx\right)^\alpha\leq C(\alpha)[Q_{M,N,n}(t)]^\alpha.
\end{align*}
Due to this and (\ref{T13.4-2}), it follows from the Sobolev embedding inequality and H\"older's inequality that
\begin{align*}
&\int_0^t\int_\Omega\theta_n^{\alpha+1}dxd\tau=\int_0^t\left\|\theta_n^{\frac{\alpha}{2}}\right\|_{L^{\frac{2(\alpha+1)}{\alpha}}
(\Omega)}^{\frac{2(\alpha+1)}{\alpha}}d\tau\nonumber\\
\leq&C\int_0^t\left(\|\theta_n^{\frac{\alpha}{2}}\|_{L^{\frac{2}{\alpha}}(\Omega)}^{\frac{1}{\alpha+1}}
\left\|\theta_n^{\frac{\alpha}{2}}\right\|_{H^1(\Omega)}^{\frac{\alpha}{\alpha+1}}\right)^{\frac{2(\alpha+1)}{\alpha}}d\tau\nonumber\\
\leq&C\int_0^t\|\theta_n\|_{L^1(\Omega)}\left(\left\|\theta_n^{\frac{\alpha}{2}}\right\|_{L^2(\Omega)}^2+\left\|\nabla\theta_n^{\frac{\alpha}{2}}
\right\|_{L^2(\Omega)}^2\right)d\tau\nonumber\\
\leq&C\int_0^t\left(\left\|\theta_n\right\|_{L^1(\Omega)}^{\alpha+1}+\|\theta_n\|_{L^1(\Omega)}
\left\|\nabla\theta_n^{\frac{\alpha}{2}}\right\|_{L^2(\Omega)}^2\right)d\tau\nonumber\\
\leq& C(\alpha)[Q_{M,N,n}(t)]^{\alpha+1},\quad\alpha\in(0, 1),
\end{align*}
or equivalently
\begin{equation}
\int_0^t\int_\Omega\theta_n^rdxd\tau\leq C(q)[Q_{M,N,n}(t)]^r,\quad r\in(1, 2). \label{T13.8}
\end{equation}
It follows from this and the H\"older's inequality that
\begin{align*}
&\int_0^t\int_\Omega|\nabla\theta_n|^{\frac{2(\alpha+1)}{3}}dxdt=\int_0^t\int_\Omega\left|\frac{2}{\alpha}\theta_n^{\frac{2-\alpha}{2}}
\nabla\theta_n^{\frac{\alpha}{2}}\right|^{\frac{2(\alpha+1)}{3}}dxdt\nonumber\\
\leq&\left(\frac{2}{\alpha}\right)^{\frac{2(\alpha+1)}{3}}\int_0^t\int_\Omega\theta_n^{\frac{(2-\alpha)(\alpha+1)}{3}}
|\nabla\theta_n^{\frac{\alpha}{2}}|^{\frac{2(\alpha+1)}{3}}dxdt\nonumber\\
\leq&\left(\frac{2}{\alpha}\right)^{\frac{2(\alpha+1)}{3}}\left(\int_0^t\int_\Omega\theta_n^{\alpha+1}dxdt\right)^{\frac{2-\alpha}{3}}
\left(\int_0^t\int_\Omega
|\nabla\theta_n^{\frac{\alpha}{2}}|^{2}dxdt\right)^{\frac{\alpha+1}{3}}\nonumber\\
\leq&C(\alpha)[Q_{M,N,n}(t)]^{\frac{2(\alpha+1)}{3}},\quad\forall\alpha\in(0,1),
\end{align*}
or equivalently
\begin{equation}
\int_0^t\int_\Omega|\nabla\theta_n|^qdxdt\leq C(q)[Q_{M,N,n}(t)]^q,\quad q\in(1,\tfrac{4}{3}). \label{T13.8-1}
\end{equation}

Combining (\ref{T13.1}), (\ref{T13.3}) (\ref{T13.8}), (\ref{T13.8-1}), and noticing that
$$
L^1(\Omega)\hookrightarrow(L^\infty(\Omega))^*\hookrightarrow\left(W^{1,\frac{q}{q-1}}_{\textmd{per}}(\Omega)\right)^*
=W^{-1,q}_{\textmd{per}}(\Omega),\quad q\in(1,\tfrac{4}{3}),
$$
one derives from equation (\ref{T12.8}) that
\begin{align}
&\|\partial_t\theta_n\|_{L^1(0, T; W^{-1, q}_{\textmd{per}}(\Omega))}\nonumber\\
\leq&C\|\nabla\theta_n-u_n\theta_n\|_{L^q(Q_T)}+C\|S_N:\nabla u_n+|\Delta d_n+\chi_M(|\nabla d_n|^2)d_n|^2\|_{L^1(Q_T)}\nonumber\\
\leq&C\left(\|\nabla\theta_n\|_{L^q(Q_T)}+\|u_n\|_{L^{\frac{40}{9}}(Q_T)}\|\theta_n\|_{L^{\frac{40q}{40-9q}}(Q_T)}\right)+C(M,N,T)\nonumber\\
\leq&C(q,M,N,T),\qquad q\in(1,\tfrac{4}{3}).  \label{T13.8-2}
\end{align}

Due to the a priori bounds (\ref{T13.1})--(\ref{T13.4}) and (\ref{T13.8-1}), there is a subsequence, still denoted by $(u_n, d_n, \theta_n, p_n)$,
such that
\begin{eqnarray*}
&&u_n\rightarrow u,~~~\mbox{ weakly in }L^{\frac{20}{9}}(0, T; W^{1, \frac{20}{9}}(\Omega)),\label{T13.12-1}\\
&&|\nabla u_n|^{\frac{2}{9}}\nabla u_n\rightarrow\overline{|\nabla u|^{\frac{2}{9}}\nabla u},~~~\mbox{ weakly in }L^{\frac{20}{11}}(Q_T),\label{T13.12-2-1}\\
&&d_n\rightarrow d,~~~\mbox{ weakly in }L^2(0, T; H^2(\Omega)),\label{T13.12-3}\\
&&\partial_t d_n\rightarrow\partial_t d,~~~\mbox{ weakly in }L^2(Q_T),\label{T13.12-4}\\
&&p_n\rightarrow p,~~~\mbox{ weakly in }L^{\frac{20}{11}}(Q_T),\label{T13.12-5}\\
&&\theta_n\rightarrow\theta,~~~\mbox{ weakly in }L^q(0, T; W^{1,q}(\Omega)),\quad\mbox{for } q\in(1,\tfrac{4}{3}),\label{T13.12-6}
\end{eqnarray*}
where $\overline{|\nabla u|^{\frac{2}{9}}\nabla u}$ denotes a vector value function, which may be different from $|\nabla u|^{\frac{2}{9}}\nabla u.$
Moreover, thanks again to the same a priori bounds, together with (\ref{T13.4-1}) and (\ref{T13.8-2}), one can apply Lemma \ref{T1lem2.2}
and Lemma \ref{T1lem2.3} and using the interpolation inequalities to conclude that
\begin{eqnarray}
&&u_n\rightarrow u,~~~\mbox{ strongly in }L^q(Q_T),~~~\mbox{ for }q\in(1, \tfrac{40}{9}),\label{T13.12}\\
&&d_n\rightarrow d,~~~\mbox{ strongly in }L^q(0, T; W^{1,q}(\Omega)),~~~\mbox{ for }q\in(1, 4),\label{T13.13}\\
&&\theta_n\rightarrow\theta,~~~\mbox{ strongly in }L^q(Q_T),~~~\mbox{ for }q\in(1, 2),\label{T13.14}
\end{eqnarray}
and thus, there is a subsequence, still denoted by $(u_n, d_n, \theta_n, p_n)$, such that
\begin{eqnarray}
&\hspace{-7mm}u_n(t)\rightarrow u(t),~~\mbox{ strongly in }L^q(\Omega),~~\mbox{ for }q\in(1, \tfrac{40}{9}),~~\mbox{ a.e. }t\in[0, T],\label{T13.14-1}\\
&\hspace{-7mm}d_n(t)\rightarrow d(t),~~\mbox{ strongly in }W^{1,q}(\Omega),~~\mbox{ for }q\in (1, 4),~~\mbox{ a.e. }t\in[0, T].\label{T13.14-2}
\end{eqnarray}
Thanks to these convergences, we can take the limit $n\rightarrow\infty$ to conclude that the limit function $(u, d, p, \theta)$ satisfies
\begin{equation}\label{T13.15}
u_t+(u\cdot\nabla)u+\nabla p=\textmd{div}(\mu(\theta)(\nabla u+\nabla u^T)+\frac{1}{N}\overline{|\nabla u|^{\frac{2}{9}}\nabla u}-\nabla d\odot\nabla d)
\end{equation}
in $\mathcal D'(Q_T)$, and
\begin{equation}\label{T13.16}
d_t+(u\cdot\nabla)d=\Delta d+\chi_M(|\nabla d|^2)d,\quad\mbox{ a.e. in }Q_T.
\end{equation}

Choosing $w=u_n(t)$ in (\ref{T12.6}), multiplying (\ref{T12.7}) by $-\Delta d_n$ and integrating over $\Omega$, then summing the resulting identities up, one may get
\begin{align*}
&\frac{1}{2}\int_\Omega(|u_n(t)|^2+|\nabla d_n(t)|^2)dx\nonumber\\
&+\int_0^t\int_\Omega\left(\frac{\mu(\theta_n)}{2}|\nabla u_n+\nabla u_n^T|^2+\frac{1}{N}|\nabla u_n|^{\frac{20}{9}}+|\Delta d_n|^2\right)dxd\tau\nonumber\\
=&-\int_0^t\int_\Omega\Delta d_n\cdot\chi_M(|\nabla d_n|^2)d_ndxd\tau+\frac{1}{2}\int_\Omega(|u_{0,n}|^2+|\nabla d_{0,n}|^2)dx,
\end{align*}
and similarly, it follows from (\ref{T13.15}) and (\ref{T13.16}) that
\begin{align*}
&\frac{1}{2}\int_\Omega(|u(t)|^2+|\nabla d(t)|^2)dx\nonumber\\
&+\int_0^t\int_\Omega\left(\frac{\mu(\theta)}{2}|\nabla u+\nabla u^T|^2+\frac{1}{N}\overline{|\nabla u|^{\frac{2}{9}}\nabla u}:\nabla u+|\Delta d|^2\right)dxd\tau\nonumber\\
=&-\int_0^t\int_\Omega\Delta d\cdot\chi_M(|\nabla d|^2)ddxd\tau+\frac{1}{2}\int_\Omega(|u_{0}|^2+|\nabla d_{0}|^2)dx.
\end{align*}
By the aid of (\ref{T13.14-1}) and (\ref{T13.14-2}), and noticing that
\begin{eqnarray*}
\int_0^t\int_\Omega\Delta d_n\cdot\chi_M(|\nabla d_n|^2)d_ndxd\tau\rightarrow\int_0^t\int_\Omega\Delta d\cdot\chi_M(|\nabla d|^2)ddxd\tau,
\end{eqnarray*}
one can deduce from the previous two identities that
\begin{align*}
&\lim_{n\rightarrow\infty}\int_0^T\int_\Omega\left(\left|\sqrt{\frac{\mu(\theta_n)}{2}}(\nabla u_n+\nabla u_n^T)\right|^2+\frac{1}{N}|\nabla u|^{\frac{20}{9}}+|\Delta d_n|^2\right)dxdt\nonumber\\
=&\int_0^T\int_\Omega\left(\left|\sqrt{\frac{\mu(\theta)}{2}}(\nabla u+\nabla u^T)\right|^2+\frac{1}{N}\overline{|\nabla u|^{\frac{2}{9}}\nabla u}:\nabla u+|\Delta d|^2\right)dxdt,
\end{align*}
from which, by Lemma \ref{T1lem2.7}, one obtains
\begin{eqnarray*}
&&\nabla u_n\rightarrow\nabla u\mbox{ strongly in }L^{\frac{20}{9}}(Q_T),\\
&&\Delta d_n\rightarrow\Delta d\mbox{ strongly in }L^2(Q_T).
\end{eqnarray*}
Thanks to these strong convergence, it is clear that $\overline{|\nabla u|^{\frac{2}{9}}\nabla u}=|\nabla u|^{\frac{2}{9}}\nabla u,$
and thus (\ref{T13.15}) reduces to
$$
u_t+(u\cdot\nabla)u+\nabla p=\textmd{div}(S_N-\nabla d\odot\nabla d),
$$
and moreover, recalling (\ref{T13.12})--(\ref{T13.14}), one can take $n\rightarrow\infty$ in (\ref{T12.8}) to derive
$$
\theta_t+u\cdot\nabla \theta=\Delta \theta+S_N:\nabla u+|\Delta d+\chi_M(|\nabla d|^2)d.
$$

Now, we are in the position to prove (\ref{T13.0-6}). Now that
$$
\theta_n\rightarrow\theta,\quad\mbox{ strongly in }L^r(Q_T),~~\mbox{ for }r\in(1,2),
$$
there is a subsequence, still denoted by $\theta_n$, such that $
\theta_n\rightarrow\theta,\mbox{ a.e. in }Q_T.
$
Noticing that
$$
u_n\rightarrow u\mbox{ strongly in }L^{\frac{20}{9}}(0, T; W^{1,\frac{20}{9}}(\Omega)),\quad d_n\rightarrow d\mbox{ strongly in }L^2(0, T; H^2(\Omega)),
$$
we have
$$
Q_{M,N,n}(t)\rightarrow Q_{M,N}(t),\qquad\text{as }n\rightarrow\infty,
$$
for all $t\in[0, T]$. Hence, it follows from  (\ref{T13.4-2}) and (\ref{T13.8-1}), Fatou's lemma,  and the weakly lower semi-continuity of norms that
\begin{align*}
\int_\Omega\theta(t)dx\leq\liminf_{n\rightarrow\infty}\int_\Omega\theta_n(t)dx=\liminf_{n\rightarrow\infty}Q_{M,N,n}(t)=Q_{M,N}(t)
\end{align*}
and
\begin{align*}
&\int_0^t\int_\Omega|\nabla\theta|^qdxds\leq\liminf_{n\rightarrow\infty}\int_0^t\int_\Omega|\nabla\theta_n|^qdxds\\
\leq& C(q)\liminf_{n\rightarrow\infty}[Q_{M,N,n}(t)]^q=C(q)[Q_{M,N}(t)]^q
\end{align*}
for $t\in[0, T]$, which completes the proof.
\end{proof}

\section{The limit $M\rightarrow\infty$}\label{T1sec5}

In the previous section, we have shown that for any given $M>0$ and $N>0$ the system (\ref{T11.6})--(\ref{T11.8}) with (\ref{T11.4})--(\ref{T11.5})
has a weak solution $(u_{M,N}, d_{M,N}, \theta_{M,N})$. In this section, we study the limit as $M$ goes to infinity
of these solutions to prove the existence of weak solutions to the following system
\begin{eqnarray}
&&u_t+(u\cdot\nabla)u+\nabla p=\textmd{div}(S_N-\nabla d\odot\nabla d),\qquad\textmd{div}u=0,\label{T14.0-1}\\
&&d_t+(u\cdot\nabla)d=\Delta d+|\nabla d|^2d,\qquad|d|=1,\label{T14.0-3}\\
&&\theta_t+u\nabla\theta=\Delta\theta+S_N:\nabla u+|\Delta d+|\nabla d|^2d|^2,\label{T14.0-5}
\end{eqnarray}
with the initial and boundary conditions (\ref{T11.4}) and (\ref{T11.5}), where $S_N$ is given by (\ref{T11.9}). In other words, we will prove the following proposition.

\begin{proposition}\label{T1prop4.1}
Assume that the $2\pi$ periodic functions $u_0, d_0$ and $\theta_0$ satisfy
$$
u_0\in L^2_\sigma(\Omega),\quad d_0\in C^2(\overline\Omega),\quad |d_0|=1,\quad \theta_0\in L^1(\Omega),\quad\inf_{x\in\Omega}\theta_0\geq\underline\theta_0
$$
for some positive constant $\underline\theta_0$. Let $E_0\geq1$ be an arbitrary constant such that
$$
\int_\Omega(|u_0|^2+|\nabla d_0|^2+\theta_0)dx\leq E_0.
$$

Then there exist two positive constants $\varepsilon_0$ and $\tau_0\in(\tfrac{1}{4}, 1)$ depending only on $e_0:=\int_\Omega(|u_0|^2+|\nabla d_0|^2)dx$, such that if
$$
\sup_{x\in\overline\Omega}\int_{B_{2R_0}(x)}(|u_0|^2+|\nabla d_0|^2)dy\leq\varepsilon_0^2,\quad\mbox{ for some }R_0\in(0, 1],
$$
then, for any $N\geq1$, the system (\ref{T14.0-1})--(\ref{T14.0-5}) with (\ref{T11.4})--(\ref{T11.5}) has a weak solution
$(u, d, \theta,p)$ in $Q_{T_0}$ with $T_0=\tau_0R_0^3$, satisfying $\inf_{(x,t)\in Q_{T_0}}\theta\geq\underline\theta_0$,
\begin{eqnarray*}
&&\sup_{0\leq t\leq T_0}\int_\Omega(|u(t)|^2+|\nabla d(t)|^2)dx+\int_{Q_{T_0}}(|\nabla u|^2+\frac{1}{N}|\nabla u|^{\frac{20}{9}}\\
&&~~~~~~~~~~~~~~~+|\Delta d|^2+|d_t|^2+|p|^{\frac{20}{11}})dxdt\leq CE_0,\\
&&\sup_{0\leq t\leq T_0}\int_\Omega\theta(t)dx\leq C_0E_0, \quad\int_{Q_{T_0}}|\nabla\theta|^qdxdt\leq CE_0^q\quad q\in(1,
\tfrac{4}{3}),
\end{eqnarray*}
and the temperature $\theta$ satisfies the following entropy inequality
$$
\partial_t\theta^\alpha+\text{div}(u\theta^\alpha)\geq\Delta\theta^\alpha+\alpha\theta^{\alpha-1}(S_N:\nabla u+|\Delta d+|\nabla d|^2d|^2)+\alpha(1-\alpha)\theta^{\alpha-2}|\nabla\theta|^2
$$
in $\mathcal D'(Q_{T_0})$ for any $\alpha\in(0, 1)$.
\end{proposition}

Several lemmas are needed to prove Proposition \ref{T1prop4.1}.

\begin{lemma}\label{T1lem4.1}
Let $(u, d, \theta, p)$ be the solution obtained in Proposition \ref{T1prop3.1}. Extend $(u, d,\theta,p)$ $2\pi$-periodically so that it is defined on $\mathbb R^2$. Then, it holds that
\begin{align*}
&\frac{d}{dt}\int_{\mathbb R^2}(|u|^2+|\nabla d|^2)\phi dx\\
&+\int_{\mathbb R^2}\left(\mu(\theta)|\nabla u+\nabla u^T|^2+\frac{1}{N}|\nabla u|^{\frac{20}{9}}+\frac{1}{2}|\Delta d|^2\right)\phi dx\\
\leq&C\int_{\mathbb R^2}[(|u|^3+|u||\nabla d|^2+|u||p|+|u||\nabla u|+\frac{1}{N}|u||\nabla u|^{\frac{11}{9}}\\
&+|\nabla d||d_t|)|\nabla\phi|+|\nabla d|^4\phi]dx,
\end{align*}
for any $\phi\in C_0^\infty({\mathbb R^2})$.
\end{lemma}

\begin{proof}
Since $(u, d, \theta, p)$ is $2\pi$ periodic with respect to space variables, one can easily check that $(u, d,\theta, p)$ is a weak solution in the whole space to the same system. Testing (\ref{T11.6}) by $u\phi$ leads to
\begin{align*}
\frac{d}{dt}\int_{\mathbb R^2}\frac{|u|^2}{2}\phi dx
=&\int_{\mathbb R^2}\left(p+\frac{|u|^2}{2}\right)(u\nabla\phi)dx-\int_{\mathbb R^2}\Big[S_N:(\nabla u\phi+u\otimes\nabla\phi)\\
&\left.+\left(\nabla\left(\frac{|\nabla d|^2}{2}\right)+\Delta d\cdot\nabla d\right)u\phi\right]dx\\
=&\int_{\mathbb R^2}\left(p+\frac{|u|^2}{2}+\frac{|\nabla d|^2}{2}\right)u\nabla\phi dx-\int_{\mathbb R^2}S_N:(\nabla u\phi+u\otimes\nabla\phi)dx\\
&-\int_{\mathbb R^2}\Delta d\cdot\nabla d\cdot u\phi dx,
\end{align*}
and thus
\begin{align*}
&\frac{d}{dt}\int_{\mathbb R^2}|u|^2\phi dx+\int_{\mathbb R^2}(\mu(\theta)|\nabla u+\nabla u^T|^2+\frac{1}{N}|\nabla u|^{\frac{20}{9}})\phi dx\nonumber\\
\leq&C\int_{\mathbb R^2}(|u|^3+|u||\nabla d|^2+|u||p|+|u||\nabla u|+\frac{1}{N}|u||\nabla u|^{\frac{11}{9}})|\nabla\phi|dx\nonumber\\
&-2\int_{\mathbb R^2}\Delta d\cdot\nabla d\cdot u\phi dx.
\end{align*}
Multiplying (\ref{T11.7}) by $-\Delta d\phi$ and integrating over $\mathbb R^2$ yields
\begin{align*}
&\frac{d}{dt}\int_{\mathbb R^2}\frac{|\nabla d|^2}{2}\phi dx+\int_{\mathbb R^2}|\Delta d|^2\phi dx\nonumber\\
=&\int_{\mathbb R^2}[\chi_M(|\nabla d|^2)d(-\Delta d)\phi-d_t\nabla d\nabla \phi]dx+\int_{\mathbb R^2}\Delta d\cdot\nabla d\cdot u\phi dx\nonumber\\
\leq& C\int_{\mathbb R^2}(|\Delta d||\nabla d|^2+|d_t||\nabla d||\nabla\phi|)dx+\int_{\mathbb R^2}\Delta d\cdot\nabla d\cdot u\phi dx.
\end{align*}
The conclusion follows from summing up these two inequalities and using Young's inequality, which completes the proof.
\end{proof}

\begin{lemma}\label{T1lem4.2}
Let $(u, d, \theta, p)$ be the weak solution obtained in Proposition \ref{T1prop3.1} with $N\geq1$. Let $e_0\geq1$ be an arbitrary constant such that
$$
\int_\Omega(|u_0|^2+|\nabla d_0|^2)dx\leq e_0.
$$

Then there is a positive constant $\varepsilon_1$ depending only on $e_0$, such that if $\varepsilon\leq\varepsilon_1$ and
$$
\sup_{x\in\Omega}\int_{B_{2R}(x)}(|u_0|^2+|\nabla d_0|^2)dy\leq\varepsilon^2,\quad\mbox{ for some }R\in(0, 1],
$$
then it holds that
\begin{eqnarray*}
&\int_{Q_{T_1}}\left(\frac{|\nabla u|^{\frac{20}{9}}}{N}+|\nabla u|^2+|\Delta d|^2+|d_t|^2+|p|^{\frac{20}{11}}\right)dxdt\nonumber\\
&\hspace{-13mm}+\sup_{0\leq t\leq T_1}\int_\Omega(|u|^2+|\nabla d|^2)dx\leq Ce_0,\nonumber\\
&\sup_{0\leq t\leq T_1}\int_\Omega\theta dx\leq\int_\Omega\theta_0dx+Ce_0,\quad \int_{Q_{T_1}}|\nabla\theta|^qdxdt\leq C\left[\left(\int_\Omega\theta_0dx\right)^q+e_0^q\right]
\end{eqnarray*}
for $q\in(1, \tfrac{4}{3}),$ where $T_1=\tau_0R^3$ with $\tau_0=\tau_0(\varepsilon,e_0)=\left(\frac{\varepsilon^4}{e_0}\right)^5$.
\end{lemma}

\begin{proof}
Let $\varepsilon\leq\varepsilon_1$, where $\varepsilon_1\leq 1
$ is a positive constant to be determined later. Extend $(u, d, \theta, p)$ periodically such that it is defined in the whole space. Testing (\ref{T11.6}) by $u$, multiplying (\ref{T11.7}) by $-\Delta d$ and integrating over $\Omega$, then summing the resulting identities up, one gets
\begin{align*}
\int_\Omega(|u(t)|^2&+|\nabla d(t)|^2)dx+\int_0^t\int_\Omega(\mu(\theta)|\nabla u+\nabla u^T|^2+\frac{2}{N}|\nabla u|^{\frac{20}{9}}+2|\Delta d|^2)dxds\nonumber\\
=&-2\int_0^t\int_\Omega\chi_M(|\nabla d|^2)d\cdot\Delta ddxds+\int_\Omega(|u_0|^2+|\nabla d_0|^2)dx\\
\leq&\int_0^t\int_\Omega(|\Delta d|^2+|\nabla d|^4)dxds+\int_\Omega(|u_0|^2+|\nabla d_0|^2)dx,
\end{align*}
from which, by Lemma \ref{T1lem2.4} (Korn's inequality), it follows that
\begin{align}
\sup_{0\leq s\leq t}\int_{\Omega}&(|u(t)|^2+|\nabla d(t)|^2)dx+\int_0^t\int_\Omega(|\nabla u|^2+\frac{1}{N}|\nabla u|^{\frac{20}{9}}+|\Delta d|^2)dxds\nonumber\\
\leq&C\int_{\Omega}(|u_0|^2+|\nabla d_0|^2)dx+C\int_0^t\int_{\Omega}(|u|^2+|\nabla d|^4)dxds,\label{T14.4}
\end{align}
for any $t\in[0, T]$.

Noticing that $u\in C([0, T]; L^2(\Omega))\mbox{ and }d\in C([0,T]; H^1(\Omega)),$
one can define $t_0$ as
\begin{equation*}
t_0=\sup\left\{t\geq0\left|\sup_{(x,s)\in Q_t}\int_{B_R(x)}(|u(t)|^2+|\nabla d(t)|^2)dy\leq 2\varepsilon^2\right.\right\},
\end{equation*}
then it is clear that
\begin{equation}\label{T14.6}
\sup_{(x,t)\in Q_{t_0}}\int_{B_R(x)}(|u|^2+|\nabla d|^2)dy=2\varepsilon^2.
\end{equation}

By Lemma \ref{T1lem2.1}, it follows from (\ref{T14.6}) that
\begin{eqnarray}
&\int_{Q_{t}}|u|^4dxds\leq C\varepsilon^2\left(\int_{Q_{t}}|\nabla u|^2dxds+\frac{t}{R^2}\sup_{0\leq t\leq t}\int_\Omega|u|^2dx\right),\label{T14.8}\\
&\int_{Q_{t}}|\nabla d|^4dxds\leq C\varepsilon^2\left(\int_{Q_{t}}|\Delta d|^2dxds+\frac{t}{R^2}\sup_{0\leq s\leq t}\int_\Omega|\nabla d|^2dx\right),\label{T14.9}
\end{eqnarray}
for any $t\in[0, t_0]$.
Substituting (\ref{T14.9}) into (\ref{T14.4}) yields
\begin{align*}
&\sup_{0\leq s\leq t}\int_{\Omega}(|u(t)|^2+|\nabla d(t)|^2)dx+\int_0^t\int_\Omega(|\nabla u|^2+\frac{1}{N}|\nabla u|^{\frac{20}{9}}+|\Delta d|^2)dxds\nonumber\\
\leq&C\varepsilon^2\int_{Q_t}|\Delta d|^2dxds+C\left(\frac{\varepsilon^2}{R^2}+1\right)t\sup_{0\leq t\leq t}\int_\Omega(|u|^2+|\nabla d|^2)dx\\
&+ C\int_\Omega(|u_0|^2+|\nabla d_0|^2)dx,
\end{align*}
and thus
\begin{align}
\sup_{0\leq s\leq t}\int_\Omega(|u|^2+|\nabla d|^2dx+&\int_{Q_{t}}(|\nabla u|^2+\frac{1}{N}|\nabla u|^{\frac{20}{9}}+|\nabla^2 d|^2)dxds\nonumber\\
\leq &C\int_\Omega(|u_0|^2+|\nabla d_0|^2)dx\leq Ce_0, \label{T14.10}
\end{align}
provided $\varepsilon_1$ is small enough and $t\leq \min\{t_0,\varepsilon^2R^2\}$.
Combining (\ref{T14.8})--(\ref{T14.10}) yields
\begin{equation}\label{T14.11}
\int_0^{t}\int_\Omega(|u|^4+|\nabla d|^4)dxds\leq C\varepsilon^2\int_\Omega(|u_0|^2+|\nabla d_0|^2)dx\leq C\varepsilon^2e_0,
\end{equation}
from which, recalling $|d|\leq1$ and using equation (\ref{T11.7}), one gets
\begin{equation}\label{T14.13}
\int_0^{t}\int_\Omega|d_t|^2dxdt\leq C\int_\Omega(|u_0|^2+|\nabla d_0|^2)dx\leq Ce_0,
\end{equation}
provide $t\leq\min\{t_0,\varepsilon^2R^2\}$.

Note that $p$ satisfies
\begin{equation*}
\left\{\begin{array}{l}
\Delta p=\textmd{div}\textmd{div}(S_N-\nabla d\odot\nabla d-u\otimes u),\quad\mbox{ in }\Omega,\\
p\mbox{ is }2\pi\mbox{ periodic },\quad \int_\Omega pdx=0.
\end{array}\right.
\end{equation*}
On account of (\ref{T14.10}) and (\ref{T14.11}), one can apply elliptic estimates to obtain
\begin{equation}\label{T14.12}
\int_0^{t}\int_\Omega|p|^{\frac{20}{11}}dxdt\leq C\left[\int_\Omega(|u_0|^2+|\nabla d_0|^2)dx+1\right]\leq Ce_0,
\end{equation}
for any $t\leq\min\{t_0,\varepsilon^2R^2\}$.

If $t_0\geq\varepsilon^2R^2$, then one can take $T_1=\varepsilon^2R^3$, and the conclusion follows from (\ref{T14.10})--(\ref{T14.12}) and Proposition \ref{T1prop3.1}. Suppose that $t_0\leq\varepsilon^2R^2$, then (\ref{T14.10})--(\ref{T14.12}) hold true for any $t\leq t_0$.
Take a cut-off function $\phi\in C_0^\infty(B_{2R}(0))$, such that $\phi\equiv1$ on $B_R(0)$, $|\nabla\phi|\leq\frac{C}{R}$ and $|\nabla^2\phi|\leq\frac{C}{R^2}$. On account of (\ref{T14.6}) and using (\ref{T14.10})--(\ref{T14.12}), one can apply Lemma \ref{T1lem4.1} to deduce that
\begin{align}
&2\varepsilon^2+\frac{1}{2}\sup_{x\in\overline\Omega}\int_0^{t_0}\int_{B_R(x)}|\Delta d|^2dydt\nonumber\\
=&\sup_{(x,t)\in Q_{t_0}}\int_{B_R(x)}(|u|^2+|\nabla d|^2)dy+\frac{1}{2}\sup_{x\in\Omega}\int_0^{t_0}\int_{B_R(x)}|\Delta d|^2dydt\nonumber\\
\leq&\sup_{(x,t)\in Q_{t_0}}\int_\Omega(|u|^2+|\nabla d|^2)\phi(x+\cdot)dy+\frac{1}{2}\sup_{x\in\Omega}\int_{Q_{t_0}}|\Delta d|^2\phi(x+\cdot)dydt\nonumber\\
\leq&\sup_{x\in\Omega}\int_{\Omega}(|u_0|^2+|\nabla d_0|^2)\phi(x+\cdot)dy+C\sup_{x\in\Omega}\int_{Q_{t_0}}[(|u|^3+|u||\nabla d|^2+|u||p|\nonumber\\
&+|u||\nabla u|+\frac{1}{N}|u||\nabla u|^{\frac{11}{9}}+|\nabla d||d_t|)|\nabla\phi(x+\cdot)|+|\nabla d|^4\phi(x+\cdot)]dydt\nonumber\\
\leq&\varepsilon^2+\frac{C}{R}\sup_{x\in\overline\Omega}\int_0^{t_0}\int_{B_{2R}(x)}(|u|^3+|u||\nabla d|^2+|u||p|+|u||\nabla u|\nonumber\\
&+\frac{1}{N}|u||\nabla u|^{\frac{11}{9}}+|\nabla d||d_t|)dydt+C\sup_{x\in\Omega}\int_0^{t_0}\int_{B_{2R}(x)}|\nabla d|^4dydt\nonumber\\
\leq&\varepsilon^2+\frac{C}{R}(R^2t_0)^{\frac{1}{5}}\left(\int_{Q_{t_0}}|u|^4dxdt\right)^{\frac{1}{4}}\left(\int_{Q_{t_0}}(N^{-\frac{20}{11}}|\nabla u|^{\frac{20}{9}}+|p|^{\frac{20}{11}})dxdt\right)^{\frac{11}{20}}\nonumber\\
&+\frac{C}{R}\left(\int_{Q_{t_0}}|u|^2dxdt\right)^{\frac{1}{2}}\left(\int_{Q_{t_0}}(|u|^4+|\nabla d|^4+|\nabla u|^2)dxdt\right)^{\frac{1}{2}}\nonumber\\
&+\frac{C}{R}\left(\int_{Q_{t_0}}|\nabla d|^2dxdt\right)^{\frac{1}{2}}\left(\int_{Q_{t_0}}|d_t|^2dxdt\right)^{\frac{1}{2}}+C\sup_{x\in\Omega}\int_0^{t_0}\int_{B_{2R}(x)}|\nabla d|^4dydt\nonumber\\
\leq&\varepsilon^2+\frac{C}{R}\left[(R^2t_0)^{\frac{1}{5}}(\varepsilon^2e_0)^{\frac{1}{4}}e_0^{\frac{11}{20}}+(e_0t_0)^{\frac{1}{2}}e_0^{\frac{1}{2}}\right]
+C\sup_{x\in\Omega}\int_0^{t_0}\int_{B_{2R}(x)}|\nabla d|^4dydt\nonumber\\
=&\varepsilon^2+C\left(\frac{t_0}{R^3}\right)^{\frac{1}{5}}\varepsilon^{\frac{1}{2}}e_0^{\frac{4}{5}}+C\left(\frac{t_0}{R^2}\right)^{\frac{1}{2}}e_0
+C\sup_{x\in\Omega}\int_0^{t_0}\int_{B_{2R}(x)}|\nabla d|^4dydt\nonumber\\
\leq&\varepsilon^2+C\left(\frac{t_0}{R^3}\right)^{\frac{1}{5}}e_0+C\sup_{x\in\Omega}\int_0^{t_0}\int_{B_{2R}(x)}|\nabla d|^4dydt,\label{T14.7}
\end{align}
provided $t_0\leq\min\{\varepsilon^2R^2, R^3\}$. We now estimate the term
$$
\sup_{x\in\Omega}\int_0^{t_0}\int_{B_{2R}(x)}|\nabla d|^4dydt.
$$
To this end, noticing that any ball of radius $2R$ can be covered by a number of balls of radius $R$, one has that
\begin{eqnarray*}
&\sup_{(x,t)\in Q_{t_0}}\int_{B_{2R}(x)}|\nabla d|^2dy\leq C\sup_{(x,t)\in Q_{t_0}}\int_{B_R(x)}|\nabla d|^2dy,\\
&\sup_{x\in\overline\Omega}\int_0^{t_0}\int_{B_{2R(x)}}|\Delta d|^2dydt\leq C\sup_{x\in\overline\Omega}\int_0^{t_0}\int_{B_R(x)}|\Delta d|^2dydt.
\end{eqnarray*}
It follows from these these estimates, (\ref{T14.6}), (\ref{T14.10}) and Ladyzhenskaya's inequality that
\begin{align*}
&\sup_{x\in\Omega}\int_0^{t_0}\int_{B_R(x)}|\nabla d|^4dydt\leq\sup_{x\in\overline\Omega}\int_0^{t_0}\int_{B_{2R}(x)}|\nabla d\phi(x+\cdot)|^4dydt\nonumber\\
\leq&C\sup_{x\in\Omega}\int_0^{t_0}\left(\int_{B_{2R}(x)}|\nabla d\phi(x+\cdot)|^2dy\right)\left(\int_{B_{2R}(x)}|\nabla(\nabla d\phi(x+\cdot))|^2dy\right)dt\nonumber\\
\leq&C\varepsilon^2\sup_{x\in\Omega}\int_0^{t_0}\int_{B_{2R}(x)}(|\nabla^2(d\phi(x+\cdot))|^2+|\nabla(d\nabla\phi(x+\cdot))|^2)dydt\nonumber\\
\leq&C\varepsilon^2\sup_{x\in\Omega}\int_0^{t_0}\int_{B_{2R}(x)}(|\Delta(d\phi(x+\cdot))|^2+|\nabla(d\nabla\phi(x+\cdot))|^2)dydt\nonumber\\
\leq&C\varepsilon^2\sup_{x\in\Omega}\int_0^{t_0}\int_{B_{2R}(x)}\left(|\Delta d|^2+\frac{|\nabla d|^2}{R^2}+\frac{1}{R^3}\right)dydt\nonumber\\
\leq&C\varepsilon^2\left(\sup_{x\in\Omega}\int_0^{t_0}\int_{B_{2R}(x)}|\Delta d|^2dydt+\frac{t_0}{R^2}\sup_{0\leq t\leq t_0}\int_\Omega|\nabla d|^2dx+\frac{t_0}{R^2}\right)\nonumber\\
\leq&C\varepsilon^2\sup_{x\in\Omega}\int_0^{t_0}\int_{B_{2R}(x)}|\Delta d|^2dydt+C\frac{t_0}{R^2}e_0,
\end{align*}
provided $t_0\leq\min\{\varepsilon^2R^2, R^3\}$. In the above, we have used elliptic estimates to the Laplace equations. Substituting this into (\ref{T14.7}) and taking $\varepsilon_1$ small enough yield
$$
\varepsilon^2\leq C\left(\frac{t_0}{R^3}\right)^{\frac{1}{5}}e_0+C\frac{t_0}{R^2}e_0\leq C\left(\frac{t_0}{R^3}\right)^{\frac{1}{5}}e_0,
$$
which gives $\varepsilon^4\leq \left(\frac{t_0}{R^3}\right)^{\frac{1}{5}}e_0$, and thus
$$
t_0\geq\left(\frac{\varepsilon^4}{e_0}\right)^5R^3,
$$
provided $t_0\leq\min\{\varepsilon^2R^2, R^3\}$. Therefore, one can take $T_1=\left(\frac{\varepsilon^4}{e_0}\right)^5R^3$ in case that $t_0\leq\varepsilon^2R^2$. Combining these two cases, one can finally choose
$$
T_1=\left\{\left(\frac{\varepsilon^4}{e_0}\right)^5, \varepsilon^2\right\}R^3=\left(\frac{\varepsilon^4}{e_0}\right)^5R^3,\quad\tau_0=\tau_0(\varepsilon,e_0)=\left(\frac{\varepsilon^4}{e_0}\right)^5,
$$
and the conclusion follows from (\ref{T14.10})--(\ref{T14.12}) by Proposition \ref{T1prop3.1}. The proof is complete.
\end{proof}

\begin{lemma}\label{T1lem4.3}
In addition to the assumptions in Proposition \ref{T1prop3.1}, we assume further that $d_0\in C^2(\overline\Omega)$. Let $(u, d, \theta, p)$ be the solution obtained in Proposition \ref{T1prop3.1}, and let $\varepsilon_1$ and $e_0$ be the same numbers as in Lemma \ref{T1lem4.2}. Then there exists a positive constant $\varepsilon_0<<\varepsilon_1$ depending only on $e_0$, such that if
$$
\sup_{x\in\Omega}\int_{B_{2R_0}(x)}(|u_0|^2+|\nabla d_0|^2)dy\leq\varepsilon_0^2,\quad\mbox{ for some }R_0\in(0,1],
$$
then it holds that
$$
\|\nabla d\|_{C^{\beta,\beta/2}(\overline\Omega\times[0, \tfrac{T_0}{2}])}\leq C(N, e_0,\|d_0\|_{C^2(\bar\Omega)})
$$
for some $\beta\in(0,1)$, where $T_0=\tau_0R_0^3$ with $\tau_0=\tau_0(\varepsilon_0, e_0)=\left(\frac{\varepsilon_0^4}{e_0}\right)^5$ as in Lemma \ref{T1lem4.2}.
\end{lemma}

\begin{proof}
The proof is similar to that of Lemma 2.1 in Lin-Lin-Wang \cite{Lin4}. It follows from (\ref{T14.11}) that
\begin{equation}\label{a}
\int_0^{T_0}\int_\Omega(|\nabla d|^4+|u|^4)dxdt\leq C\varepsilon_0^2e_0.
\end{equation}

Let $0<T<\frac{T_0}{2}$ and pick arbitrary $z_0=(x_0, t_0)\in\Omega\times[T, \tfrac{T_0}{2}]$. For any $0<R\leq\tfrac{T}{2}$ and $z_1=(x_1,t_1)\in Q_R(z_0)$, one can decompose $d=d_1+d_2$, with $d_1$ being the unique solution to
\begin{equation*}
\left\{\begin{array}{l}
\partial_td_1-\Delta d_1=0,\quad\mbox{ in }Q_R(z_1),\\
d_1|_{\partial_pQ_R(z_1)}=d,
\end{array}
\right.\end{equation*}
while $d_2$ solving the following problem
\begin{equation*}
\left\{\begin{array}{l}
\partial_t d_2-\Delta d_2=\chi_M(|\nabla d|^2)d-(u\cdot\nabla)d,\quad\mbox{ in }Q_R(z_1),\\
d_2|_{\partial_pQ_{R}(z_1)}=0.
\end{array}
\right.\end{equation*}
For $d_1$, it holds that
\begin{eqnarray}
&\hspace{-6mm}\int_{Q_r(z_1)}|\nabla d_1-(\nabla d_1)_r|^4dxdt\leq C\left(\frac{r}{R}\right)^8\int_{Q_R(z_1)}|\nabla d_1-(\nabla d_1)_R|^4dxdt,\label{T14.17-2}\\
&\int_{Q_r(z_1)}|\nabla d_1|^4dxdt\leq C\left(\frac{r}{R}\right)^4\int_{Q_R(z_1)}|\nabla d_1|^4dxdt,\label{T14.17-3}
\end{eqnarray}
for any $0<r<R\leq\tfrac{T}{2}$, where
$$
(\nabla d_1)_R=\frac{1}{|Q_R(z_1)|}\int_{Q_R(z_1)}\nabla d_1dxdt.
$$
For $d_2$, multiplying the equations of $d_2$ by $-\Delta d_2$ and integrating over $B_R(x_1)\times(t_1-R^2, t_1)$ yield
\begin{align*}
\sup_{t_1-R^2\leq t\leq t_1}&\int_{B_{R}(x)}|\nabla d_2(t)|^2dx+\int_{Q_R(z_1)}|\Delta d_2|^2dxdt\\
\leq& C\int_{Q_R(z_1)}(|\nabla d|^4+|u|^2|\nabla d|^2)dxdt,
\end{align*}
which, by Ladyzhenskaya's inequality, gives
\begin{align}
\int_{Q_R(z_1)}|\nabla d_2|^4dxdt\leq&C\left(\sup_{t_1-R^2\leq t\leq t_1}\int_{B_R(x_1)}|\nabla d_2|^2dx\right)\int_{Q_R(z_1)}|\Delta d_2|^2dxdt\nonumber\\
\leq&C\left[\int_{Q_R(z_1)}(|\nabla d|^4+|u|^2|\nabla d|^2)dxdt\right]^2\nonumber\\
\leq&C\int_{Q_R(z_1)}|\nabla d|^4dxdt\int_{Q_R(z_1)}(|u|^4+|\nabla d|^4)dxdt.\label{T14.17-4}
\end{align}
It follows from (\ref{T14.17-2})--(\ref{T14.17-4}) and (\ref{a}) that
\begin{align}
&\int_{Q_r(z_1)}|\nabla d|^4dxdt\leq C\int_{Q_R(z_1)}(|\nabla d_1|^4+|\nabla d_2|^4)dxdt\nonumber\\
\leq& C\left(\frac{r}{R}\right)^4\int_{Q_R(z_1)}|\nabla d_1|^4dxdt+C\int_{Q_R(z_1)}|\nabla d_2|^4dxdt\nonumber\\
\leq& C\left(\frac{r}{R}\right)^4\int_{Q_R(z_1)}|\nabla d|^4dxdt+C\int_{Q_R(z_1)}|\nabla d_2|^4dxdt\nonumber\\
\leq&C\left(\frac{r}{R}\right)^4\int_{Q_R(z_1)}|\nabla d|^4dxdt+C\int_{Q_R(z_1)}|\nabla d|^4dxdt\int_{Q_R(z_1)}(|u|^4+|\nabla d|^4)dxdt\nonumber\\
\leq&C\left[\left(\frac{r}{R}\right)^4+\varepsilon_0^2e_0\right]\int_{Q_R(z_1)}|\nabla d|^4dxdt,\label{T14.17-5}
\end{align}
and
\begin{align}
&\int_{Q_r(z_1)}|\nabla d-(\nabla d)_r|^4dxdt\nonumber\\
\leq& C\int_{Q_r(z_1)}(|\nabla d_1-(\nabla d_1)_r|^4+|\nabla d_2-(\nabla d_2)_r|^4)dxdt\nonumber\\
\leq&\left(\frac{r}{R}\right)^8\int_{Q_R(z_1)}|\nabla d_1-(\nabla d_1)_R|^4dxdt+C\int_{Q_R(z_1)}|\nabla d_2-(\nabla d_2)_R|^4dxdt\nonumber\\
\leq&\left(\frac{r}{R}\right)^8\int_{Q_R(z_1)}|\nabla d-(\nabla d)_R|^4dxdt+C\int_{Q_R(z_1)}|\nabla d_2-(\nabla d_2)_R|^4dxdt\nonumber\\
\leq&C\int_{Q_R(z_1)}(|u|^4+|\nabla d|^4)dxdt\int_{Q_R(z_1)}|\nabla d|^4dxdt\nonumber\\
&+C\left(\frac{r}{R}\right)^8\int_{Q_R(z_1)}|\nabla d-(\nabla d)_R|^4dxdt, \label{T14.17-6}
\end{align}
for any $0<r<R\leq\tfrac{T}{2}$. Choose  $0<\theta_0<<1,$ and $\varepsilon_0<<\varepsilon_1$ such that $2C\theta_0^4\leq\theta_0^{4\alpha}$ with $\alpha\in(\tfrac{1}{10}, 1)$ and $\varepsilon_0^2e_0\leq\theta_0^4$.
Then it follows from (\ref{T14.17-5}) that
$$
\int_{Q_{\theta_0R}(z_1)}|\nabla d|^4dxdt\leq\theta_0^{4\alpha}\int_{Q_R(z_1)}|\nabla d|^4dxdt,\quad \forall 0<R<\frac{T}{2},
$$
from which, by iterating, one can obtain
$$
\int_{Q_r(z_1)}|\nabla d|^4dxdt\leq C\left(\frac{r}{R}\right)^{4\alpha}\int_{Q_R(z_1)}|\nabla d|^4dxdt,\quad\forall0<r<R\leq\frac{T}{2},
$$
and thus, recalling (\ref{a}), we get
\begin{equation}
\int_{Q_R(z_2)}|\nabla d|^4dxdt\leq C(T, e_0)R^{4\alpha},\quad\forall0<R\leq \frac{T}{2}.\label{T14.17-7}
\end{equation}
By the estimates in Lemma \ref{T1lem4.2} and using the Gagliado-Nirenberg inequality, one can check easily that
$$
\int_{Q_{T_0}}|u|^{\frac{40}{9}}dxdt\leq C(N,e_0)
$$
and thus, it follows from the H\"older's inequality that
\begin{align*}
\int_{Q_R(z_1)}|u|^4dxdt\leq\left(\int_{Q_R(z_1)}|u|^{\frac{40}{9}}dxdt\right)^{{\frac{9}{10}}}R^{\frac{2}{5}}\leq C(N, e_0)R^{\frac{2}{5}}.
\end{align*}
Substituting this inequality and (\ref{T14.17-7}) into (\ref{T14.17-6}) leads to
\begin{align*}
\int_{Q_r(z_1)}|\nabla d-(\nabla d)_r|^4dxdt\leq& C\left(\frac{r}{R}\right)^8\int_{Q_R(z_1)}|\nabla d-(\nabla d)_R|^4dxdt\\
&+C(N,T,e_0)R^{4\alpha+{\frac{2}{5}}}
\end{align*}
for any $0<r<R\leq \frac{T}{2}$. From which, by iterating, one gets
\begin{align*}
\int_{Q_r(z_1)}|\nabla d-(\nabla d)_r|^4dxdt\leq& C(N,T,e_0)\left(\frac{r}{R}\right)^{4\alpha+{\frac{2}{5}}}\int_{Q_R(z_1)}|\nabla d-(\nabla d)_R|^4dxdt\\
&+C(N,T,e_0)r^{4\alpha+{\frac{2}{5}}},
\end{align*}
for any $0<r<R\leq\tfrac{T}{2}$, and thus
$$
\int_{Q_r(z_1)}|\nabla d-(\nabla d)_R|^4dxdt\leq C(N,T,e_0)R^{4\alpha+{\frac{2}{5}}},\quad\forall 0<R\leq\frac{T}{2}.
$$
This implies that $\nabla d\in C^{\beta,\beta/2}(\overline\Omega[T,T_0])$, and
$$
\|\nabla d\|_{C^{\beta,\beta/2}(\bar\Omega\times[T, \tfrac{T_1}{2}])}\leq C(N,T,e_0)
$$
for some $\beta\in(0,1)$.

Similarly, we can deduce the $C^{\beta,\beta/2}$ estimates up to the initial time, and finally obtain
$$
\|\nabla d\|_{C^{\beta,\beta/2}(\bar\Omega\times[T, \tfrac{T_0}{2}])}\leq C(N,e_0,\|d\|_{C^2(\bar\Omega)}),
$$
which yields the desired conclusion.
\end{proof}

Now, we can prove Proposition \ref{T1prop4.1}.

\textbf{Proof of Proposition \ref{T1prop4.1}.} Denote by $\{(u_M, d_M, \theta_M, p_M)\}$ the weak solutions obtained by Proposition \ref{T1prop3.1}. By Lemma \ref{T1lem4.2} and Lemma \ref{T1lem4.3}, there exist two positive constants $\varepsilon_0$ depending only on $e_0$ and $\tau_0=\tau_0(\varepsilon_0, e_0)$, such that
\begin{eqnarray}
&\sup_{0\leq t\leq T_0}\int_\Omega(|u_M|^2+|\nabla d_M|^2)dx+\int_{Q_{T_0}}(|\nabla u_M|^2+\frac{1}{N}|\nabla u_M|^{\frac{20}{9}}\nonumber\\
&+|\Delta d_M|^2+|\partial_td_M|^2+|p_M|^{\frac{20}{11}})dxdt\leq CE_0,\label{T14.18}\\
&\|\nabla d_M\|_{C^{\beta,\frac{\beta}{2}}(\bar\Omega\times[0, \tfrac{T_0}{2}])}\leq C(N, E_0),\quad\inf_{(x, t)\in Q_{T_0}}\theta_M\geq\underline\theta_0,\label{T14.19-1}\\
&\vspace{-8mm}\sup_{0\leq t\leq T_0}\int_\Omega\theta_M(t)dx\leq CE_0,\quad\int_{Q_{T_0}}|\nabla\theta_M|^qdxdt\leq CE_0^q,~~ q\in(1, \frac{4}{3}),\label{T14.21}
\end{eqnarray}
where $T_0=\tau_0R_0^3$.

By the Gagliado-Nirenberg inequality, we can deduce from (\ref{T14.18}) and (\ref{T14.21}) that
\begin{equation}\label{T14.22}
\int_0^{T_0}\int_\Omega(|u_M|^{\frac{40}{9}}+|\nabla d_M|^4+\theta_M^r)dxdt\leq C(r, E_0),\quad r\in(1,2).
\end{equation}

In view of the estimates in the above, there is a subsequence, denoted still by $\{(u_M,d_M,$ $\theta_M, p_M)\}$, and $(u, d, \theta, p)$, such that
\begin{eqnarray}
&&u_M\rightarrow u,\quad\mbox{ weakly in }L^{\frac{20}{9}}(0,T_0; W^{1,\frac{20}{9}}(\Omega)),\label{T14.25}\\
&&|\nabla u_M|^{\frac{2}{9}}\nabla u_M\rightarrow\overline{|\nabla u|^{\frac{2}{9}}\nabla u},\quad\mbox{ weakly in }L^{\frac{20}{11}}(Q_{T_0}),\label{T14.25-1}\\
&&d_M\rightarrow d,\quad\mbox{ weakly in }L^2(0, T_0; H^2(\Omega)),\label{T14.26}\\
&&d_M\rightarrow d,\quad\mbox{ in }C^1(\overline{Q_{T_0}}),\label{T14.26-0}\\
&&p_M\rightarrow p,\quad\mbox{ weakly in }L^{\frac{20}{11}}(Q_{T_0}),\label{T12.26-1}\\
&&\theta_M\rightarrow\theta,\quad\mbox{ weakly in }L^q(0, T_0; W^{1,q}(\Omega)),\quad q\in (1, \frac{4}{3}),\label{T14.27}
\end{eqnarray}
and the corresponding estimates (\ref{T14.18})--(\ref{T14.21}) hold true for $(u, d, \theta, p)$.

Due to the a priori bounds (\ref{T14.18})--(\ref{T14.21}), similar to (\ref{T13.12})--(\ref{T13.14}), one can show that
\begin{eqnarray}
&&u_M\rightarrow u,\quad\mbox{ strongly in }L^q(Q_{T_0}),\quad\mbox { for }q\in(1, \frac{40}{9}),\label{T14.31}\\
&&d_M\rightarrow d,\quad\mbox{ strongly in }L^q(0, T_0; W^{1,q}(\Omega)),\quad\mbox{ for }q\in(1,4 ),\label{T14.32}\\
&&\theta_M\rightarrow\theta,\quad\mbox{ strongly in }L^r(Q_{T_0}),\quad\mbox{ for }r\in(1,2),\label{T14.33}
\end{eqnarray}
and $(u, d, \theta, p)$ satisfies the system
\begin{eqnarray}
&&u_t+(u\cdot\nabla)u+\nabla p=\textmd{div}\left(\mu(\theta)(\nabla u+\nabla u^T)\right.\nonumber\\
&&~~~~~~~~~~~~~~~~~+\frac{1}{N}\overline{|\nabla u|^{\frac{2}{9}}\nabla u}-\nabla d\odot\nabla d),\label{T14.33-1}\\
&&\textmd{div}u=0,\nonumber\\
&&d_t+(u\cdot\nabla) d=\Delta d+|\nabla d|^2d.\label{T14.33-3}
\end{eqnarray}
Moreover, using equation (\ref{T14.33-3}) and the initial condition $|d_0|=1$, by maximal principle for parabolic equations, it holds that $|d|=1$ in $Q_{T_0}$.

Testing (\ref{T11.6}) by $u_M$, multiplying equation (\ref{T11.7}) by $-\Delta d_M$ and integrating over $\Omega$, then summing up the resulting identities, one gets that
\begin{align}
\int_\Omega&\left(\frac{|u_M(t)|^2}{2}+\frac{|\nabla d_M(t)|^2}{2}\right)dx+\int_{Q_t}\left(\frac{\mu(\theta_M)}{2}|\nabla u_M+\nabla u_M^T|^2\right.\nonumber\\
&\left.+\frac{1}{N}|\nabla u_M|^{\frac{20}{9}}+|\Delta d_M|^2\right)dxds\nonumber\\
=&-\int_{Q_t}\chi_M(|\nabla d_M|^2)d_M\cdot\Delta d_Mdxds+\int_\Omega\left(\frac{|u_0|^2}{2}+\frac{|\nabla d_0|^2}{2}\right)dx\label{T14.33-4}
\end{align}
for $t\in[0, T_0]$.
Similarly, it follows from (\ref{T14.33-1}) and (\ref{T14.33-3}) that
\begin{align}
\int_\Omega&\left(\frac{|u(t)|^2}{2}+\frac{|\nabla d(t)|^2}{2}\right)dx+\int_{Q_t}\left(\frac{\mu(\theta)}{2}|\nabla u+\nabla u|^2\right.\nonumber\\
&\left.+\frac{1}{N}\overline{|\nabla u|^{\frac{2}{9}}\nabla u}:\nabla u+|\Delta d|^2\right)dxds\nonumber\\
=&-\int_{Q_t}|\nabla d|^2d\cdot\Delta ddxds+\int_\Omega\left(\frac{|u_0|^2}{2}+\frac{|\nabla d_0|^2}{2}\right)dx\label{T14.33-5},
\end{align}
for $t\in[0, T_0]$.
Since $u_M\rightarrow u$ strongly in $L^q(Q_{T_0})$ for $1<q<\frac{40}{9}$ and $\nabla d\rightarrow\nabla d$ in $C(\overline{Q_{T_0}})$, it holds that
\begin{eqnarray*}
&&\int_\Omega(|u_M(t)|^2+|\nabla d_M(t)|^2)dx\rightarrow\int_\Omega(|u(t)|^2+|\nabla d(t)|^2)dx,~~\mbox{ a.e. }t\in[0, T_0],\label{T14.33-6}\\
&&\int_{Q_t}\chi_M(|\nabla d_M|^2)d_M\cdot\Delta d_Mdxds\rightarrow\int_{Q_t}|\nabla d|^2d\cdot\Delta ddxds.\label{T14.33-7}
\end{eqnarray*}
It follows from this and (\ref{T14.33-4}) and (\ref{T14.33-5}) that
\begin{align*}
&\lim_{n\rightarrow\infty}\int_0^t\int_\Omega\left(\left|\sqrt{\frac{\mu(\theta_M)}{2}}(\nabla u_M+\nabla u_M^T)\right|^2+\frac{1}{N}|\nabla u_M|^{\frac{20}{9}}+|\Delta d_M|^2\right)dxds\nonumber\\
=&\int_0^t\int_\Omega\left(\left|\sqrt{\frac{\mu(\theta)}{2}}(\nabla u+\nabla u^T)\right|^2+\frac{1}{N}\overline{|\nabla u|^{\frac{2}{9}}\nabla u}:\nabla u+|\Delta d|^2\right)dxds,
\end{align*}
from which, one can apply Lemma \ref{T1lem2.7} to infer that
\begin{eqnarray*}
&&\nabla u_M\rightarrow \nabla u,\quad\mbox{ strongly in }L^{\frac{20}{9}}(Q_{T_0}),\label{T14.33-8}\\
&&\Delta d_M\rightarrow \Delta d,\quad\mbox{ strongly in }L^2(Q_{T_0}),\label{T14.33-9}
\end{eqnarray*}
and thus, by taking limit $M\rightarrow\infty$, we can derive
$$
\theta_t+u\nabla\theta=\Delta\theta+S_N:\nabla u+|\Delta d+|\nabla d|^2d|^2,\quad\mbox{ in }\mathcal D'(Q_{T_0}).
$$

Finally, we show that the temperature $\theta$ satisfies the entropy inequality in the sense of distribution. Let $w_\varepsilon$ be the standard modifier and $\theta_\varepsilon=\theta*w_\varepsilon$. Then it holds that
\begin{equation}
\partial\theta_\varepsilon+\text{div}(u\theta_\varepsilon)=\Delta\theta_\varepsilon+(S_N:\nabla u+|\Delta d+|\nabla d|^2d|^2)*w_\varepsilon+r_\varepsilon,\label{T1A.1}
\end{equation}
where $r_\varepsilon=\text{div}(u\theta_\varepsilon)-\text{div}(u\theta)*w_\varepsilon$. Recall that $u\in L^{\frac{20}{9}}(0, T_0; W^{1,\frac{20}{9}}(\Omega))$ and $\theta\in L^r(Q_{T_0})$ for $r\in(1, 2)$. It follows from  Lemma 2.3 in Lions \cite{Lions} that
\begin{equation}\label{T1A.2}
r_\varepsilon\rightarrow0, \quad\mbox{ in }L^{\frac{20r}{9r+20}}(Q_{T_0})\mbox{ as }\varepsilon\rightarrow0,\quad\mbox{ for }r\in (\tfrac{20}{11}, 2).
\end{equation}
Multiplying equation (\ref{T1A.1}) by $\alpha\theta_\varepsilon^{\alpha-1}$ yields
\begin{align}
\partial_t\theta_\varepsilon^\alpha+\text{div}(u\alpha_\varepsilon^\alpha)=&\Delta\theta_\varepsilon^\alpha+\alpha\theta_\varepsilon^{\alpha-1}(S_N:\nabla u+|\Delta d+|\nabla d|^2d|^2)*w_\varepsilon\nonumber\\
&+\alpha(1-\alpha)\theta_\varepsilon^{\alpha-2}|\nabla\theta_\varepsilon|^2+\alpha r_\varepsilon\theta_\varepsilon^{\alpha-1}.\label{T1A.3}
\end{align}
Due to (\ref{T1A.2}) and $\theta\geq\underline\theta_0>0$, it holds that
\begin{eqnarray}
&&\theta_\varepsilon^\alpha\rightarrow\theta^\alpha,\quad\mbox{ in }L^{\frac{r}{\alpha}}(Q_{T_0}),\label{T1A.4}\\
&&u\theta_\varepsilon^\alpha\rightarrow u\theta^\alpha,\quad\mbox{ in }L^{\frac{4r}{4\alpha+r}}(Q_{T_0}),\label{T1A.5}\\
&&\alpha r_\varepsilon\theta_\varepsilon^{\alpha-1}\rightarrow0,\quad\mbox{ in }\mathcal D'(Q_{T_0}),\label{T1A.6}\\
&&(S_N:\nabla u+|\Delta d+|\nabla d|^2d|^2)*w_\varepsilon\theta_\varepsilon^{\alpha-1}\nonumber\\
&&\qquad\rightarrow(S_N:\nabla u+|\Delta d+|\nabla d|^2d|^2)\theta^{\alpha-1},\quad\mbox{ in }\mathcal D'(Q_{T_0}).\label{T1A.7}
\end{eqnarray}
In (\ref{T1A.7}), we have used the fact that
$$
(S_N:\nabla u+|\Delta d+|\nabla d|^2d|^2)*w_\varepsilon\rightarrow (S_N:\nabla u+|\Delta d+|\nabla d|^2d|^2),\quad\mbox{ in }L^1(Q_{T_0}),
$$
and that $\theta_\varepsilon^{\alpha-1}$ is uniformly bounded and converges to $\theta^{\alpha-1}$ a.e. in $Q_{T_0}$. For the term $\theta_\varepsilon^{\alpha-2}|\nabla\theta_\varepsilon|^2$, since $\theta_\varepsilon\rightarrow\theta$ in $L^r(Q_{T_0})$ for $r\in(1,2)$, it follows that $\nabla\theta_\varepsilon^{\frac{\alpha}{2}}\rightarrow\nabla\theta^{\frac{\alpha}{2}}$ weakly in $L^2(Q_{T_0})$, and consequently, by the weakly lower semicontinuity of norms, for any $0\leq\varphi\in C_0^\infty(Q_{T_0})$, we have
\begin{align*}
&\int_0^{T_0}\int_\Omega\theta^{\alpha-2}|\nabla\theta|^2\varphi dxdt=\left(\frac{2}{\alpha}\right)^2\int_0^{T_0}\int_\Omega|\nabla\theta^{\frac{\alpha}{2}}\sqrt\varphi|^2dxdt\nonumber\\
\leq&\liminf_{\varepsilon\rightarrow0}\left(\frac{2}{\alpha}\right)^2\int_0^{T_0}\int_\Omega|\nabla\theta_\varepsilon^{\frac{\alpha}{2}}\sqrt\varphi|^2dxdt=
\liminf_{\varepsilon\rightarrow0}\int_0^{T_0}\int_\Omega\theta_\varepsilon^{\alpha-2}|\nabla\theta_\varepsilon|^2\varphi dxdt.
\end{align*}
On account of this, together with (\ref{T1A.4})--(\ref{T1A.7}), one can take the limit in (\ref{T1A.3}) to conclude that the entropy inequality holds true in the sense of distribution. The proof is completed.

\section{The limit $N\rightarrow\infty$ and the local existence}\label{T1sec6}

We have shown in the previous section that, for any given $N$, the system (\ref{T14.0-1})--(\ref{T14.0-5}) with (\ref{T11.4})--(\ref{T11.5}) has a weak solution $(u_N, d_N, \theta_N, p_N)$. We will study the limit of such a sequence as $N$ goes to infinity to prove the local existence of weak solutions to the original system (\ref{T11.1})--(\ref{T11.5}). In particular, we will prove the following result.

\begin{proposition}\label{T1prop5.1}
(Local existence) Assume that the $2\pi$ periodic functions $u_0, d_0$ and $\theta_0$ satisfy
$$
u_0\in L^2_\sigma(\Omega),\quad d_0\in H^1(\Omega;S^2),\quad \theta_0\in L^1(\Omega),\quad\inf_{x\in\overline\Omega}\theta_0\geq\underline\theta_0
$$
for some positive constant $\underline\theta_0$. Let $E_0\geq1$ be an arbitrary constant such that
$$
\int_\Omega(|u_0|^2+|\nabla d_0|^2+\theta_0)dx\leq E_0.
$$
Let $\varepsilon_0$ and $\tau_0$ be the constants stated in Proposition \ref{T1prop4.1}. Suppose that
$$
\sup_{x\in\overline\Omega}\int_{B_{2R_0}(x)}(|u_0|^2+|\nabla d_0|^2)dy\leq\varepsilon_0^2,\quad\mbox{ for some }R_0\in(0, 1].
$$

Then there exists a weak solution $(u, d,\theta,p)$ to the system (\ref{T11.1})--(\ref{T11.5}) on $Q_{T_0}=\Omega\times(0, T_0)$ with $T_0=\tau_0R_0^3$, satisfying  $\inf_{(x,t)\in Q_{T_0}}\theta\geq\underline\theta_0$,
\begin{eqnarray*}
&\sup_{0\leq t\leq T_0}\int_\Omega(|u|^2+|\nabla d|^2)dx+\int_{Q_{T_0}}(|\nabla u|^2
+|\Delta d|^2+|\partial_td|^2+|p|^2)dxdt\leq CE_0,\\
&\sup_{0\leq t\leq T_0}\int_\Omega\theta_N(t)dx\leq CE_0,\quad\int_{Q_{T_0}}|\nabla\theta_N|^qdxdt\leq CE_0^q,\quad q\in(1, \frac{4}{3}),
\end{eqnarray*}
and
\begin{eqnarray*}
&\int_\Omega\left(\frac{|u(t)|^2}{2}+\frac{|\nabla d(t)|^2}{2}+\theta(t)\right)dx=\int_\Omega\left(\frac{u_0|^2}{2}+\frac{|\nabla d_0|^2}{2}+\theta_0\right)dx,
\end{eqnarray*}
for any $t\in[0, T_0]$, and the temperature $\theta$ satisfies the following entropy inequality
$$
\partial\theta^\alpha+\text{div}(u\theta^\alpha)\geq\Delta\theta^\alpha+\alpha\theta^{\alpha-1}(S:\nabla u+|\Delta d+|\nabla d|^2d|^2)+\alpha(1-\alpha)\theta^{\alpha-2}|\nabla\theta|^2
$$
in $\mathcal D'(Q_{T_0})$ for $\alpha\in(0, 1)$.
\end{proposition}

\begin{proof}
Since the torus $\mathbb T^2$ can be viewed as a two dimensional compact surface without boundary. By Lemma \ref{T1lem2.8}, there is a sequence of $2\pi$ periodic functions $\{d_{0,N}\}\subseteq C^\infty(\overline\Omega; S^2)$, such that
$$
d_{0,N}\rightarrow d_0,\quad\mbox{ strongly in }H^1(\Omega).
$$
Due to the absolute continuity of integrals, for $\varepsilon_0$ there is a $R_0$ independent of $N$, such that
$$
\sup_{x\in\overline\Omega}\int_{B_{2R_0}(x)}(|u_0|^2+|\nabla d_{0,N}|^2)dy\leq2\varepsilon_0^2,\quad\mbox{ for any }N.
$$
By Proposition \ref{T1lem4.1}, there is a weak solution $\{(u_N, d_N, \theta_N, p_N)\}$ on $Q_{T_0}$ to the system (\ref{T14.0-1})--(\ref{T14.0-5})
with initial data $(u_0, d_{0,N},\theta_0)$, such that $\inf_{(x,t)\in Q_{T_0}}\theta_N\geq\underline\theta_0$,
\begin{align}
&\sup_{0\leq t\leq T_0}\int_\Omega(|u_N|^2+|\nabla d_N|^2)dx+\int_{Q_{T_0}}(|\nabla u_N|^2+\frac{1}{N}|\nabla u_N|^{\frac{20}{9}}\nonumber\\
&\qquad \qquad+|\Delta d_N|^2+|\partial_td_N|^2+|p_N|^{\frac{20}{11}})dxdt\leq CE_0,\label{T15.2}\\
&\sup_{0\leq t\leq T_0}\int_\Omega\theta_N(t)dx\leq CE_0,\quad\int_{Q_{T_0}}|\nabla\theta_N|^qdxdt\leq CE_0^q,\quad q\in(1, \frac{4}{3}).\label{T15.4}
\end{align}

Note that (\ref{T15.2}) implies
$$
\frac{1}{N}|\nabla u_N|^{\frac{2}{9}}\nabla u_N\rightarrow0\quad\mbox{ in }L^{\frac{20}{11}}(Q_{T_0}).
$$
Using (\ref{T15.2})--(\ref{T15.4}) and (\ref{T14.0-1})--(\ref{T14.0-3}), one can argue as in the proof of Proposition
\ref{T1prop4.1} to show that
\begin{eqnarray}
&&u_N\rightarrow u\quad\mbox{ strongly in }L^q(Q_{T_0}),\quad q\in(1, 4),\label{T15.7}\\
&&d_N\rightarrow d\quad\mbox{ strongly in }L^q(0, T_0; W^{1,q}(\Omega)),\quad q\in(1,4),\label{T15.8}\\
&&\theta_N\rightarrow\theta\quad\mbox{ strongly in }L^r(Q_{T_0}),\quad r\in(1,2),\label{T15.9}
\end{eqnarray}
and $(u, d,\theta, p)$ satisfies $\inf_{(x,t)\in Q_{T_0}}\theta\geq\underline\theta_0$,
\begin{align}
&u_t+(u\cdot\nabla)u+\nabla p=\textmd{div}(\mu(\theta)(\nabla u+\nabla u^T)-\nabla d\odot\nabla d),\quad\textmd{div}u=0,\label{T1C.1}\\
&d_t+(u\cdot\nabla)d=\Delta d+|\nabla d|^2d,\qquad|d|=1.\label{T1C.2}
\end{align}
Besides, inequality (\ref{T15.4}) also holds true for $\theta$, while (\ref{T15.2}) is replaced by
\begin{equation*}
\sup_{0\leq t\leq T_0}\int_\Omega(|u|^2+|\nabla d|^2)dx+\int_{Q_{T_0}}(|\nabla u|^2
+|\Delta d|^2+|\partial_td|^2+|p|^2)dxdt\leq CE_0.\label{T15.10}
\end{equation*}

Testing (\ref{T14.0-1}) by $u_N$ and (\ref{T14.0-3}) by $-\Delta d_N$, respectively, summing up the resulting identities, and using the identity
\begin{equation*}
(\Delta d_n+|\nabla d_N|^2d_N)\cdot\Delta d_N=|\Delta d_N+|\nabla d_N|^2d_N|^2
\end{equation*}
which follows from $|d_N|=1$, one gets
\begin{align*}
\int_\Omega&\left(\frac{|u_0|^2}{2}+\frac{|\nabla d_{0,N}|^2}{2}\right)dx=\int_\Omega\left(\frac{|u_N(t)|^2}{2}+\frac{|\nabla d_N(t)|^2}{2}\right)dx\nonumber\\
&+\int_0^t\int_\Omega\left(\frac{\mu(\theta_N)}{2}|\nabla u_N+\nabla u_N^T|^2+\frac{|\nabla u_N|^{\frac{20}{9}}}{N}+|\Delta d_N+|\nabla d_N|^2d_N|^2\right)dxds, \end{align*}
for any $t\in[0, T_0]$. Similarly, one deduces from equations (\ref{T1C.1}) and (\ref{T1C.2}) that
\begin{align}
\int_\Omega&\left(\frac{|u_0|^2}{2}+\frac{|\nabla d_{0}|^2}{2}\right)dx=\int_\Omega\left(\frac{|u(t)|^2}{2}+\frac{|\nabla d(t)|^2}{2}\right)dx\nonumber\\
&+\int_0^t\int_\Omega\left(\frac{\mu(\theta)}{2}|\nabla u+\nabla u^T|^2+|\Delta d+|\nabla d|^2d|^2\right)dxds,\label{B.3}
\end{align}
for any $t\in[0, T_0]$. Combining these two identities with (\ref{T15.7}) and (\ref{T15.8}), we can apply Lemma \ref{T1lem2.7} to infer that
$$
\nabla u_N\rightarrow\nabla u\quad\mbox{and}\quad\Delta d_N\rightarrow\Delta d,\quad\mbox{strongly in }L^2(Q_{T_0}).
$$
Therefore, one can take the limit $N\rightarrow\infty$ to conclude that
$$
\theta_t+\text{div}(u\theta)=\frac{\mu(\theta)}{2}|\nabla u+\nabla u^T|^2+|\Delta d+|\nabla d|^2d|^2\quad\mbox{ in }\mathcal D'(Q_{T_0}),
$$
from which, by a standard regularization argument, one deduces
$$
\int_\Omega\theta(t)dx+\int_0^t\int_\Omega\left(\frac{\mu(\theta)}{2}|\nabla u+\nabla u^T|^2+|\Delta d+|\nabla d|^2d|^2\right)dxds=\int_\Omega\theta_0dx,
$$
for any $t\in[0, T_0]$. This, combined with (\ref{B.3}), gives
$$
\int_\Omega\left(\frac{|u(t)|^2}{2}+\frac{|\nabla d(t)|^2}{2}+\theta(t)\right)dx=\int_\Omega\left(\frac{|u_0|^2}{2}+\frac{|\nabla d_0|^2}{2}+\theta_0\right)dx,
$$
for any $t\in[0, T_0]$.

Finally, we verify the entropy inequality for the temperature. For arbitrary test function $0\leq\varphi\in C_0^\infty(Q_{T_0})$, by Proposition \ref{T1prop4.1},
it holds that
\begin{align}
&\int_0^{T_0}\int_\Omega(\theta_N^\alpha\varphi_t+u\theta_N^\alpha\nabla\varphi)dxdt\nonumber\\
\leq&\int_{Q_{T_0}}\nabla\theta_N^\alpha\nabla\varphi dxdt-\int_{Q_{T_0}}\Big[\alpha\theta_N^{\alpha-1}\Big(\frac{\mu(\theta_N)}{2}|\nabla u_N+\nabla u_N^T|^2+\frac{1}{N}|\nabla u|^{\frac{20}{9}}\nonumber\\
&+|\Delta d_N+|\nabla d_N|^2d_N|^2\Big)+\alpha(1-\alpha)\theta_N^{\alpha-2}|\nabla\theta_N|^2\Big]\varphi dxdt.\label{T1B.1}
\end{align}
Note that
\begin{eqnarray*}
&\vspace{-5mm}\sqrt{\theta_N^{\alpha-2}\mu(\theta_N)}(\nabla u_N+\nabla u_N^T)\rightarrow\sqrt{\theta^{\alpha-2}\mu(\theta)}(\nabla u+\nabla u^T),\quad\mbox{ weakly in }L^2(Q_{T_0}),\\
&\vspace{-5mm}\sqrt{\theta_N^{\alpha-1}}(\Delta d_N+|\nabla d_N|^2d_N)\rightarrow\sqrt{\theta^{\alpha-1}}(\Delta d+|\nabla d|^2d),\mbox{ weakly in }L^2(Q_{T_0}),\\
&\vspace{-5mm}\sqrt{\theta_N^{\alpha-1}}\nabla\theta_N\rightarrow\sqrt{\theta^{\alpha-1}}\nabla\theta,\quad\mbox{ weakly in }L^2(Q_{T_0}).
\end{eqnarray*}
It follows from the weakly lower semicontinuity of norms and (\ref{T1B.1}) that
\begin{align*}
&\int_0^{T_0}\int_\Omega(\theta^\alpha\varphi_t+u\theta^\alpha\nabla\varphi)dxdt\nonumber\\
\leq&\int_0^{T_0}\int_\Omega\nabla\theta^\alpha\nabla\varphi dxdt
-\liminf_{N\rightarrow\infty}\int_0^{T_0}\int_\Omega\left[\alpha\theta_N^{\alpha-1}\left(\frac{\mu(\theta_N)}{2}|\nabla u_N+\nabla u_N^T|^2\right.\right.\nonumber\\
&\left.\left.+\frac{1}{N}|\nabla u|^{\frac{20}{9}}+|\Delta d_N+|\nabla d_N|^2d_N|^2\right)+\alpha(1-\alpha)\theta_N^{\alpha-2}|\nabla\theta_N|^2\right]\varphi dxdt\\
\leq&\int_0^{T_0}\int_\Omega\nabla\theta^\alpha\nabla\varphi dxdt
-\int_\Omega[\alpha\theta^{\alpha-1}(S:\nabla u+|\Delta d+|\nabla d|^2d|^2\\
&+\alpha(1-\alpha)\theta^{\alpha-2}|\nabla\theta|^2]\varphi dxdt,
\end{align*}
which shows that the entropy inequality holds true in the sense of distribution.
The proof is completed.
\end{proof}

\section{The global existence}\label{T1sec7}

In this section, we will show the global existence of weak solution to the system (\ref{T11.1})--(\ref{T11.5}) and thus
establish Theorem \ref{T1them1.2}.

\begin{lemma}\label{T1lem7.1}
Let $0=T_0<T_1<T_2<\cdots<T_N<T_{N+1}=\infty$ and the $2\pi$ periodic functions $u, d$ and $\theta$ satisfy
\begin{eqnarray*}
&&u\in L^\infty(0, T; L^2)\cap L^2(0, T; H^1(\Omega))\cap C([0, T]; L^2_w),\\
&&d\in L^\infty(0, T; H^1)\cap L^2(T_i, T_{i+1}-\delta; H^2)\cap L^2(T_N, T; H^2)\cap C([0, T]; H^1_w),\\
&&\theta\in L^\infty(0, T; L^1)\cap L^q(0, T; W^{1,q})\cap C([0,T]; W^{-1,q}_{per}),\quad q\in(1, \frac{4}{3}),
\end{eqnarray*}
for any small $\delta>0$ and $T>T_N$, $0\leq i\leq N-1$. Suppose that $(u, d, \theta)$ satisfies the system (\ref{T11.1})--(\ref{T11.3}) and the entropy inequality \begin{equation}\label{T18.0}
\partial_t\theta^\alpha+\textmd{div}(u\theta^\alpha)\geq\Delta\theta^\alpha+\alpha\theta^{\alpha-1}(S:\nabla u+|\Delta d+|\nabla d|^2d|^2)+\alpha(1-\alpha)\theta^{\alpha-2}|\nabla\theta|^2
\end{equation}
for any $\alpha\in(0, 1)$, in the sense of $\mathcal D'(\Omega\times(T_i, T_{i+1})$, $0\leq i\leq N$.

Then $(u, d, \theta)$ satisfies the system (\ref{T11.1})--(\ref{T11.3}) and the entropy inequality in the sense of $\mathcal D'(Q_T)$.
\end{lemma}

\begin{proof}
Let $\varphi\in C_0^\infty(\Omega\times(0, T_2))$ with $\textmd{div}\varphi=0$ be any given test function, and $\chi(s)\in C_0^\infty((-1,1))$ such that $\chi\equiv1$ on $(-1/2, 1/2)$. For any $\varepsilon>0$, define $\chi_\varepsilon(s)=\chi\left(\frac{s-T_1}{\varepsilon}\right)$.  Then $\chi_\varepsilon\in C_0^\infty((T_1-\varepsilon, T_1+\varepsilon))$ and $\chi_\varepsilon\equiv1$ on $(T_1-\varepsilon/2, T_1+\varepsilon/2)$. Recalling that $(u, d, \theta)$ is a weak solution in the time interval $(0, T_1)\cup(T_1, T_2)$, we can deduce that
\begin{align*}
&\int_0^{T_2}\int_\Omega(u\varphi_t+u\otimes u:\nabla\varphi-(S+\sigma^{\text{nd}}):\nabla\varphi)dxdt\\
=&\int_0^{T_2}\int_\Omega(u(\varphi\chi_\varepsilon)_t+u\otimes u:\nabla(\varphi\chi_\varepsilon)-(S+\sigma^{\text{nd}}):\nabla(\varphi\chi_\varepsilon))dxdt\\
=&\int_0^{T_2}\int_\Omega(u\varphi_t+u\otimes u:\nabla\varphi-(S+\sigma^{\text{nd}}):\nabla\varphi)\chi_\varepsilon dxdt\\
&+\int_0^{T_2}\int_\Omega u\varphi\chi_\varepsilon'(t)dxdt=I_1+I_2.
\end{align*}
It follows from the definition of $\chi_\varepsilon$ and the regularities of $(u, d)$ that
\begin{align*}
|I_1|=&\left|\int_{T_1-\varepsilon}^{T_1+\varepsilon}\int_\Omega(u\varphi_t+u\otimes u:\nabla\varphi-(S+\sigma^{\text{nd}}):\nabla\varphi)\chi_\varepsilon dxdt\right|\\
\leq&C\int_{T_1-\varepsilon}^{T_1+\varepsilon}\int_\Omega(|u|+|u|^2+|\nabla u|+|\nabla d|^2)dxdt\rightarrow0,\quad\mbox{ as }\varepsilon\rightarrow0.
\end{align*}
For $I_2$, one has
\begin{align*}
I_2=\int_{T_1-\varepsilon}^{T_1+\varepsilon}\int_\Omega u\varphi\frac{1}{\varepsilon}\chi'\left(\frac{t-T_1}{\varepsilon}\right)dxdt=\int_{T_1-\varepsilon}^{T_1+\varepsilon}f(t)\frac{1}{\varepsilon}\chi'
\left(\frac{t-T_1}{\varepsilon}\right)dxdt,
\end{align*}
where
$$
f(t)=\int_\Omega u\varphi dx.
$$
Since $u\in C([0, T]; L^2_w)$, thus $f\in C([0, T])$. It follows that
\begin{align*}
|I_2|=&\left|\int_{T_1-\varepsilon}^{T_1+\varepsilon}f(t)\frac{1}{\varepsilon}\chi'
\left(\frac{t-T_1}{\varepsilon}\right)dt\right|\\
=&\left|\int_{T_1-\varepsilon}^{T_1+\varepsilon}(f(t)-f(T_1))\frac{1}{\varepsilon}\chi'
\left(\frac{t-T_1}{\varepsilon}\right)dt\right|\\
\leq&C\max_{t\in[T_1-\varepsilon, T_1+\varepsilon]}|f(t)-f(T_1)|\rightarrow0,\quad\mbox{ as }\varepsilon\rightarrow0,
\end{align*}
since $f(t)$ is continuous on $[0, T]$. Combining the above statements shows that
\begin{equation*}\label{T17.1}
\int_0^{T_2}\int_\Omega(u\varphi_t+u\otimes u:\nabla\varphi-(S+\sigma^{\text{nd}}):\nabla\varphi)dxdt=0
\end{equation*}
for any function $\varphi\in C_0^\infty(\Omega\times(0, T_2))$ with $\text{div}\varphi=0$. Similarly, one can show
 that the above identity holds true for any $\varphi\in C_0^\infty(\Omega\times(0, T))$ with $\textmd{div}\varphi=0$.
 Similar arguments as above show that
$(u, d, \theta)$ satisfies (\ref{T11.2}), (\ref{T11.3}) and the entropy inequality (\ref{T18.0}) in the sense of $\mathcal D'(Q_T)$. The proof is completed.
\end{proof}

\textbf{Proof of Theorem \ref{T1them1.2}.} The proof is divided into several steps as follows.

\textbf{Step 1. Extend  the local solution to the first "singular" time.} Set
$$
\mathcal E_0=\mathcal E_{0,1}+\int_\Omega\theta_0dx,\quad\mathcal E_{0,1}=\int_\Omega\left(\frac{|u_0|^2}{2}+\frac{|\nabla d_0|^2}{2}\right)dx.
$$
Let $\varepsilon_0$ and $\tau_0$ be the constants stated in Proposition \ref{T1prop5.1}. By the absolute continuity of integrals, we can define $R_0$ as
$$
R_0=\sup\left\{R\left|\sup_{x\in\Omega}\int_{B_{2R}(x)}(|u_0|^2+|\nabla d_0|^2)dy\leq\varepsilon_0^2, R\in(0, 1]\right.\right\}.
$$
Then one can verify easily, by the definition of $R_0$, that
$$
\sup_{x\in\Omega}\int_{B_{2R_0}(x)}(|u_0|^2+|\nabla d_0|^2)dy\leq\varepsilon_0^2.
$$

By Proposition \ref{T1prop5.1}, there is a weak solution $(u, d, \theta, p)$ in $Q_{t_1}=\Omega\times(0,\tau_0R_0^3)$ to the system (\ref{T11.1})--(\ref{T11.5}), satisfying
\begin{equation}\label{T16.1}
\int_\Omega\left(\frac{|u(t)|^2}{2}+\frac{|\nabla d(t)|^2}{2}+\theta(t)\right)dx=\mathcal E_0
\end{equation}
for $t\in(0, t_1)$,
\begin{eqnarray*}
&&u\in L^\infty(0, t_1; L^2)\cap L^2(0, t_1; H^1),\quad p\in L^2(Q_{t_1}),\\
&&d\in L^\infty(0, t_1; H^1)\cap L^2(0, t_1; H^2),\\
&&\theta\in L^\infty(0, t_1; L^1)\cap L^q(0, t_1; W^{1,q}),\quad q\in(1,\frac{4}{3}),\\
&&\inf_{(x, t)\in Q_{t_1}}\theta\geq\underline\theta_0,
\end{eqnarray*}
and the following entropy inequality holds true
\begin{equation}\label{T16.1-1}
\partial_t\theta^\alpha+\textmd{div}(u\theta^\alpha)\geq\Delta \theta^\alpha+\alpha\theta^{\alpha-1}(S:\nabla u+|\Delta d+|\nabla d|^2d|^2)+\alpha(1-\alpha)\theta^{\alpha-2}|\nabla\theta|^2
\end{equation}
in $\mathcal D'(\Omega\times(0, t_1))$.

 It follows from (\ref{T11.1}), (\ref{T11.2}),(\ref{T11.3}), and  the above regularities of $(u, d, \theta, p)$ that
$$
u\in C([0, t_1]; L^2),\quad d\in C([0, t_1]; H^1),\quad\theta\in C([0, t_1]; W^{-1,q}_{per}(\Omega)),\quad q\in(1, \frac{4}{3}),
$$
and thus $(u, d, \theta)$ can be uniquely defined at $t=t_1$.

We extend $(u, d, \theta)$ to another time $t_2>t_1$ as follows. Define
$$
R_1=\sup\left\{R\left|\sup_{x\in\Omega}\int_{B_{2R}(x)}(|u(t_1)|^2+|\nabla d(t_1)|^2)dy\leq\varepsilon_0^2, R\in(0, 1]\right.\right\},
$$
and thus, by the definition of $R_1$, it holds that
$$
\sup_{x\in\Omega}\int_{B_{2R_1}(x)}(|u(t_1)|^2+|\nabla d(t_1)|^2)dy\leq\varepsilon_0^2.
$$
Recalling (\ref{T16.1}), by Proposition \ref{T1prop5.1}, we can extend $(u, d, \theta)$ to the time
$$
t_2=\tau_0(R_0^3+R_1^3),
$$
such that
\begin{eqnarray*}
&&u\in L^\infty(t_1, t_2; L^2)\cap L^2(t_1, t_2; H^1)\cap C([t_1, t_2]; L^2),\quad p\in L^2(\Omega\times(t_1, t_2),\\
&&d\in L^\infty(t_1, t_2; H^1)\cap L^2(t_1, t_2; H^2)\cap C([t_1, t_2]; H^1),\\
&&\theta\in L^\infty(t_1, t_2; L^1)\cap L^q(t_1, t_1; W^{1,q})\cap C([t_1, t_2]; W^{-1,q}_{\text{per}}),\quad q\in(1,\frac{4}{3}),\\
&&\inf_{(x, t)\in \Omega\times[t_1, t_2]}\theta\geq\underline\theta_0,
\end{eqnarray*}
the energy identity (\ref{T16.1}) holds true for $t\in[0, t_2)$, and the entropy inequality (\ref{T16.1-1}) holds true in $\mathcal D'(\Omega\times(t_1, t_2))$. Applying Lemma \ref{T1lem7.1}, we infer that $(u, d, \theta)$ is a weak solution in $Q_{t_2}$, and the entropy inequality (\ref{T16.1-1}) holds true in $\mathcal D'(Q_{t_2})$.

Repeat the above procedure, we obtain two sequences $\{R_i\}_{i=0}^\infty$ and $\{t_i\}_{i=0}^\infty$ with $t_0=0$
$$
t_i=\tau_0(R_0^3+R_1^3+\cdots+R_i^3),\quad i=1,2,\cdots,
$$
such that $(u, d, \theta, p)$ can be extended to be a weak solution in $Q_{t_i}$ with
\begin{eqnarray*}
&&u\in L^\infty(0, t_i; L^2)\cap L^2(0, t_i; H^1)\cap C([0, t_i]; L^2),\quad p\in L^2(Q_{t_i}),\\
&&d\in L^\infty(0, t_i; H^1)\cap L^2(0, t_i; H^2)\cap C([0, t_i]; H^1),\\
&&\theta\in L^\infty(0, t_i; L^1)\cap L^q(0, t_i; W^{1,q})\cap C([0, t_i]; W^{-1,q}_{\text{per}}),\quad q\in(1,\frac{4}{3}),\\
&&\inf_{(x, t)\in Q_{t_i}}\theta\geq\underline\theta_0,
\end{eqnarray*}
the energy identity (\ref{T16.1}) holds true for $t\in[0, t_i)$, the entropy inequality (\ref{T16.1-1}) holds true in $\mathcal D'(\Omega\times(0, t_i))$, and
$$
R_i=\sup\left\{R\in(0,1]\left|\sup_{x\in\Omega}\int_{B_{2R}(x)}(|u(t_i)|^2+|\nabla d(t_i)|^2)dy\leq\varepsilon_0^2\right.\right\}.
$$

Set
$$
T_1=\sum_{i=0}^\infty\tau_0R_i^3.
$$
Note that
$$
u\in L^\infty(0, T; L^2)\cap L^2(0, T; H^1),\quad d\in L^\infty(0, T; H^1)\cap L^2(0, T; H^2)
$$
for any $T<T_1$. Testing (\ref{T11.1}) by $u$, multiplying equation (\ref{T11.2}) by $-\Delta d$ and integrating over $\Omega$, then summing up the resulting identities,  one gets
\begin{align}
\int_\Omega(|u(t)|^2+|\nabla d(t)|^2)dx+&2\int_s^t\int_\Omega(S:\nabla u+|\Delta d+|\nabla d|^2d|^2)dxd\tau\nonumber\\
=&\int_\Omega(|u(s)|^2+|\nabla d(s)|^2)dx\label{T16.4}
\end{align}
for all $0\leq s\leq t<T_1$. Note that the entropy inequality (\ref{T16.1-1}) holds true in $\mathcal D'(\Omega\times(0, T_1))$. Thus,  by a standard regularization argument, it follows from (\ref{T16.1-1}) and (\ref{T16.4}) that
$$
\int_0^{T_1}\int_\Omega|\nabla\theta^{\frac{\alpha}{2}}|^2dxdt\leq C\mathcal E_0^\alpha,\qquad \alpha\in(0,1),
$$
from which, using (\ref{T16.1}) and the fact that $\inf_{(x,t)\in Q_{T_1}}\theta\geq\underline\theta_0>0$, similar to (\ref{T13.8-1}), one can deduce
\begin{equation}\label{T16.4-1}
\int_0^{T_1}\int_\Omega|\nabla\theta|^qdxdt\leq C\mathcal E_0^q,\qquad q\in(1, \frac{4}{3}).
\end{equation}

If $T_1=\infty$, then we have obtained a weak solution $(u, d, \theta, p)$ defined for all finite time. Suppose that $T_1<\infty,$ then $\lim_{i\rightarrow\infty}R_i=0$. Take an integer $i_0$ such that $R_i\leq\frac{1}{2}$ for $i\geq i_0$. Then for each $i\geq i_0$, we can pick $x_i\in\bar\Omega$, such that
\begin{equation}\label{T16.2}
\int_{B_{2R_i}(x_i)}(|u(t_i)|^2+|\nabla d(t_i)|^2)dy\geq\frac{\varepsilon_0^2}{2}.
\end{equation}
In fact, suppose that this is not true, then there is $i'\geq i_0$, such that
$$
\sup_{x\in\Omega}\int_{B_{2R_{i'}}(x)}(|u(t_{i'})|^2+|\nabla d(t_{i'})|^2)dy\leq\frac{\varepsilon_0^2}{2}.
$$
Due to this, by the absolute continuity of integrals, we can find $R'>R_{i'}$, such that
$$
\sup_{x\in\Omega}\int_{B_{2R'}(x)}(|u(t_{i'})|^2+|\nabla d(t_{i'})|^2)dy\leq{\varepsilon_0^2},
$$
contradicting to the definition of $R_{i'}$, and hence (\ref{T16.2}) holds true. Let $x_0\in\overline \Omega$ be the limit of $\{x_i\}$ (possibly by taking subsequence). Then for any $r>0$, it holds that
\begin{equation}\label{T16.3}
B_{2R_i}(x_i)\subseteq B_r(x_0)\quad\mbox{ for large }i.
\end{equation}
By the aid of (\ref{T16.2}) and (\ref{T16.3}), for any given $r>0$, it holds that
\begin{align}
&\sup_{(x, t)\in\overline{Q_{T_1}}}\int_{B_r(x)}(|u(t)|^2+|\nabla d(t)|^2)dy\nonumber\\
\geq&\limsup_{i\rightarrow\infty}\int_{B_{2R_i}(x_i)}(|u(t_i)|^2+|\nabla d(t_i)|^2)dy\geq\frac{\varepsilon_0^2}{2}.\label{C.1}
\end{align}

Note that (\ref{T11.1}) can be rewritten as
$$
u_t+\textmd{div}(u\otimes u)+\nabla\left(p+\frac{|\nabla d|^2}{2}\right)=\textmd{div}S-(\Delta d+|\nabla d|^2d)\cdot\nabla d,
$$
from which, using (\ref{T16.4}), one can deduce easily that $u\in C([0, T_1]; L^2_w(\Omega))$. It follows from (\ref{T11.2}), (\ref{T11.3}), and (\ref{T16.4}) that $d\in C([0, T_1]; H^1_w(\Omega))$ and $\theta\in C([0, T_1]; W^{-1, q}_{\text{per}}(\Omega))$ for $q\in(1, \frac{4}{3})$. Therefore, we can define uniquely the value of $(u, d, \theta)$ as the limit of $(u(t), d(t), \theta(t))$ at $T_1$ as $t\rightarrow T_1^{-}$.
By the aid of (\ref{T16.2}) and (\ref{T16.3}), it follows from weakly lower semi-continuity of norms that
\begin{align*}
\tilde{\mathcal E_1}(T_1):=&\int_{\tilde\Omega}(|u(T_1)|^2+|\nabla d(T_1)|^2)dx\\
=&\lim_{r\rightarrow0}\int_{\tilde\Omega\setminus B_r(x_0)}(|u(T_1)|^2+|\nabla d(T_1)|^2)dx\\
\leq&\lim_{r\rightarrow0}\liminf_{i\rightarrow\infty}\int_{\tilde\Omega\setminus B_r(x_0)}(|u(t_i)|^2+|\nabla d(t_i)|^2)dx\\
\leq&\lim_{r\rightarrow0}\liminf_{i\rightarrow\infty}\int_{\tilde\Omega\setminus B_{2R_i}(x_i)}(|u(t_i)|^2+|\nabla d(t_i)|^2)dx\\
\leq&\tilde{\mathcal E_1}(0)-\frac{\varepsilon_0^2}{2},
\end{align*}
where $\tilde\Omega=(-2\pi,2\pi)\times(-2\pi,2\pi)$ and $(u,d)$ is extended $2\pi$ periodically to the whole space.

\textbf{Step 2. Extend the local solution to be a global one.} Set $T_0=0$. In the same way as in Step 1, starting from $T_1$, we can extend $(u, d, \theta)$ to another time $T_2>T_1$, such that $(u, d, \theta)$ is weak solution in the time intervals $(0, T_1)\cup(T_1, T_2)$,
\begin{eqnarray*}
&&u\in L^\infty(0, T_2; L^2)\cap L^2(0, T_2; H^1)\cap C([0, T_2]; L^2_w),\\
&&d\in L^\infty(0, T_2; H^1)\cap L^2((0, T_1-\delta)\cup(T_1, T_2-\delta); H^2)\cap C([0, T_2]; H^1_w),\\
&&\theta\in L^\infty(0, T_2; L^1)\cap L^q(0, T_2; W^{1,q})\cap C([0,T_2]; W^{-1,q}_{per}),\quad q\in(1, \frac{4}{3}),
\end{eqnarray*}
for any $\delta>0$, the energy identity (\ref{T16.1}) holds true for $t\in[0, T_2)$, inequality (\ref{T16.4-1}) holds true with $T_1$ replaced by $T_2$, the entropy inequality (\ref{T16.1-1}) holds true in $\mathcal D'(\Omega\times(T_i, T_{i+1}))$ and identity (\ref{T16.4}) holds true for $T_i\leq s\leq t<T_{i+1}$, $i=0, 1$.

If $T_2=\infty$, then by Lemma \ref{T1lem7.1}, $(u, d, \theta)$ is a global weak solution. Otherwise, it has $0<T_2<\infty$. In this case, $T_2$ can be characterized by (\ref{C.1}) with $T_1$ replaced by $T_2$, and
$$
\tilde{\mathcal E_1}(T_2)\leq\tilde{\mathcal E_1}(0)-\varepsilon_0^2.
$$

In general, if we obtain $0<T_1<T_2<\cdots<T_k<\infty$, such that:

(i) $(u, d, \theta)$ is a weak solution in the time intervals $\cup_{i=0}^{k-1}(T_i, T_{i+1})$,
\begin{eqnarray*}
&&u\in L^\infty(0, T_k; L^2)\cap L^2(0, T_k; H^1)\cap C([0, T_k]; L^2_w),\\
&&d\in L^\infty(0, T_k; H^1)\cap L^2(\cup_{i=0}^{k-1}(T_i, T_{i+1}-\delta); H^2)\cap C([0, T_k]; H^1_w),\\
&&\theta\in L^\infty(0, T_k; L^1)\cap L^q(0, T_k; W^{1,q})\cap C([0,T_k]; W^{-1,q}_{per}),\quad q\in(1, \frac{4}{3}),
\end{eqnarray*}
for any $\delta>0;$

(ii) the energy identity (\ref{T16.1}) holds true for $t\in[0, T_k)$, inequality (\ref{T16.4-1}) holds true with $T_1$ being replaced by $T_k$, the entropy inequality (\ref{T16.1-1}) holds true in $\mathcal D'(\Omega\times(T_i, T_{i+1}))$ and identity (\ref{T16.4}) holds true for $T_i\leq s\leq t<T_{i+1}$, $i=0, 1,\cdots, k$;

(iii) $\tilde{\mathcal E_1}(t)$ is continuously decreasing on $[0, T_1)\cup[T_1, T_2)\cup\cdots\cup[T_{k-1}, T_k)$ and
$$
\tilde{\mathcal E_1}(T_k)\leq\tilde{\mathcal E_1}(0)-\frac{k}{2}\varepsilon_0^2.
$$
Then,  starting from $T_k$, we can extend $(u, d, \theta)$ to another time $T_{k+1}>T_k$ in the same way as in Step 1, such that (i) and (ii) hold true with $k$ being replaced by $k+1$. If $T_{k+1}=\infty$, then by Lemma \ref{T1lem7.1}, we have obtained the global weak solution. If $T_{k+1}<\infty$, then
$$
\tilde{\mathcal E_1}(T_{k+1})\leq \tilde{\mathcal E_1}(0)-\frac{k+1}{2}\varepsilon_0^2.
$$

Since $\tilde{\mathcal E_1}(0)$ is a finite number, there are at most finite many such $T_i$'s. Denote by $T_N$ the largest one of these $T_i$'s.
Then starting from $T_N$, we can extend $(u, d, \theta)$ to any finite time $T$ in the same way as Step 1. By Lemma \ref{T1lem7.1}, we have obtained the
global weak solution and proved all the conclusions stated in Theorem \ref{T1them1.2} except the $C^{\beta, \beta/2}$ regularity of $d$. While the
$C^{\beta, \beta/2}$ regularity of $d$ can be shown in the same way as in Lemma \ref{T1lem4.3} as long as we observe that $u\in L^4(Q_T)$,
$\nabla d\in L^4(\Omega\times[T_i, T_{i+1}-\delta))$ for any $\delta>0$ and $i=0, 1, \cdots, N-1$ and $\nabla d\in L^4(\Omega\times[T_N, T))$ which guarantee
the smallness condition used in Lemma \ref{T1lem4.3}. Thus the proof is completed.

\section*{Acknowledgments}
{This research is supported
partially by Zheng Ge Ru Foundation, Hong Kong RGC Earmarked
Research Grants CUHK4041/11P and CUHK4048/13P, a Focus Area Grant from The
Chinese University of Hong Kong, and a grant from Croucher Foundation. }
\par

\end{document}